\documentclass[reqno,11pt]{amsart}
\usepackage{geometry}
\usepackage{mathtools,amssymb,amsthm,mathrsfs,color,lineno,paralist,graphicx,float}
\usepackage[colorlinks,
linkcolor=blue,
anchorcolor=green,
citecolor=blue, 
]{hyperref}
\usepackage{tikz}
\usepackage{graphicx}
\usetikzlibrary{arrows.meta}

\usepackage{enumitem}
\usepackage[T1]{fontenc}
\usepackage[utf8]{inputenc}
\usepackage{subfig} 
\usepackage[justification = centering, labelsep =period]{caption} 
\usepackage{multirow}

\usepackage{calc}
\definecolor{bleu1}{RGB}{0,57,128}
\def\bleu1{\color{bleu1}}
\usepackage{amsthm}
\usepackage{etoolbox}
\patchcmd{\section}{\normalfont}{\normalfont \bleu1}{}{}
\patchcmd{\subsection}{\normalfont}{\normalfont \bleu1}{}{}
\patchcmd{\subsubsection}{\normalfont}{\normalfont \bleu1}{}{}

\newtheorem{theorem}{Theorem}[section]
\newtheorem{proposition}{Proposition}[section]
\newtheorem{lemma}{Lemma}[section]
\newtheorem{corollary}{Corollary}[section]
\newtheorem{remark}{Remark}[section]

\newtheorem{problem}{Problem}[section]
\newtheorem{conjecture}{Conjecture}[section]

\title[Square-Root Profiles and Sharp Fourier Decay]{Square-Root Profiles and Sharp Fourier Decay
of Quasiperiodic Spectral Measures}

\author{Zhenfu Wang}
\address{
	Chern Institute of Mathematics and LPMC, Nankai University, Tianjin 300071, China
}
\email{zhenfuwang@mail.nankai.edu.cn}

\author{Qi Zhou}
\address{
	Chern Institute of Mathematics and LPMC, Nankai University, Tianjin 300071, China
}
\email{qizhou@nankai.edu.cn}

\date{}

\begin{document}

\begin{abstract}
We study spectral measures of one-dimensional quasiperiodic
Schr\"odinger operators. For subcritical  Schr\"odinger operator   with  Diophantine frequency, we
prove sharp measurable-set bounds, answering a question of Avila and
Jitomirskaya \cite{AJ};
we prove optimal Fourier decay and the sharp
decay order of the time-averaged quantum return probability, confirming
a conjecture of Aloisio, Carvalho, and de Oliveira \cite{ACO} and a prediction of
Zhong and Mosseri \cite{ZhongMosseri}.
For the subcritical almost Mathieu operator, we establish summable
square-root envelopes for the spectral density, with an optimal
density-one formulation of the lower bound, thereby answering a
question of Avila \cite{Aqi}. 

These results follow from a general theorem for
reflectionless measures under a summable square-root Craig-type
condition. We establish an exact residue expansion of the squared
Borel transform, which connects local density blow-up with long-time
averaged return probabilities: the endpoint residues determine both
the square-root edge profiles and the leading coefficient of the
time-averaged return probability.
\end{abstract}

\maketitle

\section{Introduction}

Spectral measures connect the spectral properties of an operator with
the dynamics that it generates.  Their local regularity quantifies the
concentration of spectral mass in energy, while the Fourier transform
of a spectral measure is the corresponding quantum return
amplitude.  Quantitative estimates on these objects relate spectral
geometry to temporal correlations and quantum transport
\cite{Last,BCM}. They also form a
natural problem in harmonic analysis: one asks how the size and
geometry of a measure determine cancellation in its Fourier transform
\cite{Mattila}.  Absolute continuity implies Fourier decay by the
Riemann-Lebesgue lemma, but it gives no rate.  The basic question of
this paper is therefore how the structure of a one-dimensional
Schr\"odinger operator and the geometry of its spectrum determine both
the concentration of its spectral measures and their long-time
Fourier behavior.

\subsection{Quasiperiodic Schr\"odinger operators and main results}

We consider the one-dimensional quasiperiodic Schr\"odinger operator
on $\ell^2(\mathbb Z)$,
$$
 (H_{v,\alpha,\theta}u)_n
 =u_{n+1}+u_{n-1}+v(\theta+n\alpha)u_n,
 \qquad n\in\mathbb Z,
$$
where $\theta\in\mathbb T=\mathbb R/\mathbb Z$ is the phase,
$\alpha\in\mathbb R\setminus\mathbb Q$ is the frequency, and
$v:\mathbb T\to\mathbb R$ is real analytic.  The principal example is
the almost Mathieu operator (AMO)
\begin{equation*}\label{schro}
 (H_{\lambda,\alpha,\theta}u)_n
 =u_{n+1}+u_{n-1}
 +2\lambda\cos\bigl(2\pi(\theta+n\alpha)\bigr)u_n.
\end{equation*}
Introduced as a model for an electron on a two-dimensional lattice in
a uniform magnetic field \cite{Pe,Ha,R}, this operator has become a
basic model at the intersection of spectral theory, dynamical systems,
and mathematical physics.  We refer to \cite{DF,DF2,J,Y} for broader
accounts of quasiperiodic Schr\"odinger operators.

Let $\Sigma_{v}$ denote the  spectrum of $H_{v,\alpha,\theta}$, which is independent of the phase $\theta$.   For
$f\in\ell^2(\mathbb Z)$, define
$$
 \mu_{\theta,f}(X)
 =\bigl\langle f,\mathbf1_X(H_{v,\alpha,\theta})f\bigr\rangle,
 \qquad
 \mu_{\theta,n}=\mu_{\theta,\delta_n}.
$$
For the canonical two-site measure we write
$ \nu_\theta=\frac12\bigl(\mu_{\theta,0}+\mu_{\theta,1}\bigr).$

Avila's global theory organizes the spectrum of a one-frequency
analytic Schr\"odinger operator through the behavior of the Lyapunov
exponent under complexification \cite{AvilaGlobal}.  If
$$
 A_E(x)=
 \begin{pmatrix}
 E-v(x)&-1\\
 1&0
 \end{pmatrix},
 \qquad
 L(E;\varepsilon) =L\bigl(\alpha,A_E(\cdot+i\varepsilon)\bigr),
$$
then an energy is called \emph{subcritical} when
$L(E;\varepsilon)=0$ for all sufficiently small
$|\varepsilon|$.  Correspondingly, Schr\"odinger operator $H_{v,\alpha,\theta}$ is subcritical if for any $E\in \Sigma_v$, $(\alpha, A_E)$ is subcritical. 
For Almost Mathieu operator, the operator is subcritical if and only if $0<\lambda<1$ \cite{AvilaGlobal}. 
The importance of  subcriticality is formulated by Avila's Almost Reducibility Conjecture \cite{avila2010almost,AvilaGlobal,Avi2023KAM}, which  yield absolutely continuous spectral measures \cite{AvilaAMO,avila2010almost,avila}. 
Absolute continuity, however, is only qualitative: it does not control
how much spectral mass may concentrate on a small measurable subset of
the energy axis.

For subcritical Schr\"odinger $H_{v,\alpha,\theta}$  with Diophantine frequencies,  Avila and Jitomirskaya \cite{AJ} proved the sharp interval estimate
$
\mu_{\theta,f}((E-r,E+r))\leq C_f r^{1/2},
$
for $f\in \ell^1(\mathbb Z)$.  Here $\alpha$ is Diophantine (denote by  $\alpha \in \mathrm{DC}:= \cup_{\tau,\kappa>0} \mathrm{DC}(\kappa,\tau)$),  where 
$$
\mathrm{DC}(\kappa,\tau)=
\left\{\alpha\in\mathbb{R}\setminus\mathbb{Q}:\| k\alpha\|_{\mathbb{R}/\mathbb{Z}}
\geq\frac{\kappa}{|k|^{\tau}}
\ \text{for all } k\in\mathbb{Z}\setminus\{0\}
\right\}.
$$

In
\cite[Remark~4.4]{AJ}, they asked whether the same exponent controls
arbitrary measurable sets.

\begin{problem}\label{prob:AJ}
Suppose that $\alpha \in \mathrm{DC}$, does one have
$ \mu_{\theta,f}(X)\le C_f|X|^{1/2}$
for every measurable set $X\subset\mathbb R$?
\end{problem}

Our first theorem gives an affirmative answer in the globally
subcritical Diophantine regime.

\begin{theorem}\label{thm:qp-holder}
Suppose that $\alpha \in \mathrm{DC}$,   Schr\"odinger operator $H_{v,\alpha,\theta}$  is subcritical.  Then there exists $C<\infty$ such that
$$
 \mu_{\theta,f}(X) \le C\|f\|_{\ell^1}^2|X|^{1/2}
$$
for every $\theta\in\mathbb T$, every $f\in\ell^1(\mathbb Z)$, and
every measurable set $X\subset\mathbb R$.
\end{theorem}

Earlier interval estimates include the exponential arithmetic regime
of Liu and Yuan \cite{LiuYuan} and the analytic multifrequency setting
of Zhao \cite{Zhao}.  Li, You, and Zhou obtained exact local
distribution results for the subcritical almost Mathieu operator and
stratified H\"older estimates at subcritical energies
\cite{LiYouZhou}.  Theorem~\ref{thm:qp-holder} is stronger than an
interval estimate: it controls every measurable set.  Equivalently,
the densities of the corresponding spectral measures satisfy uniform
weak-$L^2$ bounds.

We next consider Fourier decay.  For a finite Borel measure $\mu$ on
$\mathbb R$, let
$$
 \widehat\mu(t)=\int_{\mathbb R}e^{-itE}\,d\mu(E).
$$
If $\mu$ is absolutely continuous, the Riemann-Lebesgue lemma gives
$\widehat\mu(t)\to0$, but gives no quantitative rate.  For a
Schr\"odinger operator $H$, the spectral theorem gives
$$
 \widehat\mu_\psi(t)=\langle\psi,e^{-itH}\psi\rangle,
$$
so $|\widehat\mu_\psi(t)|^2$ is the return probability of the state
$\psi$.  Fourier decay thus measures both cancellation in energy and
long-time recurrence in quantum dynamics
\cite{Last,BCM,KetzmerickPetschelGeisel}.

Periodic operators provide the reference case: Floquet theory reduces
their propagators to oscillatory integrals; see \cite{DFY} for the
optimal known uniform dispersive estimate.  For aperiodic models the
picture is subtler.  Sturmian operators have singular continuous
spectral measures and fractal spectra.  For the Fibonacci Hamiltonian
at small coupling, Leclerc proved polynomial Fourier decay for the
density-of-states measure, and hence for the phase-averaged return
amplitude \cite{Leclerc}. 

For small analytic quasiperiodic potentials with Diophantine
frequency, Bambusi and Zhao obtained uniform dispersive estimates
\cite{BambusiZhao}.  For long-range crystals, Leclerc, Sabri, and Sahlsten constructed
examples with arbitrarily fast polynomial dispersion, driven by
oscillatory cancellation in Weierstrass-type Floquet functions
\cite{LeclercSabriSahlsten}.
These bound all matrix elements of the
propagator, including the diagonal elements, which are the Fourier
transforms of the single-site spectral measures.
Their squared moduli are the corresponding quantum return
probabilities. For the subcritical almost Mathieu operator,
numerical work of Zhong and Mosseri \cite{ZhongMosseri}
motivated the following conjecture of Aloisio, Carvalho, and
de Oliveira \cite{ACO} on the precise decay of the time-averaged
return probability.

\begin{conjecture}[{\cite[Conjecture~1.3]{ACO}}]
\label{conj:ACO}
For subcritical  AMO, one has
$$
 \frac1T\int_0^T
 \left|\left\langle \delta_0,e^{-itH_{\lambda,\alpha,\theta}}\delta_0 \right\rangle
 \right|^2\,dt \sim\frac{\log T}{T},
 \qquad T\to\infty.
$$
\end{conjecture}

In the following, we confirm  this conjecture  for
Diophantine frequencies.
In fact, we will prove a stronger result: that the Fourier decay rate is $t^{-1/2}$:

\begin{theorem}\label{thm:qp-fourier}
Under the assumptions of Theorem~\ref{thm:qp-holder}, there exist
$0<c\le C<\infty$ such that, for every $\theta\in\mathbb T$ and
$n\in\mathbb Z$,
$$
 |\widehat\mu_{\theta,n}(t)| \le C(1+|t|)^{-1/2},
 \qquad  \limsup_{|t|\to\infty}
 |t|^{1/2}|\widehat\mu_{\theta,n}(t)|\ge c.
$$
Moreover,
$$
\lim_{T\to\infty}\frac{1}{\log T}\int_0^T|\widehat\mu_{\theta,n}(t)|^2\,dt>0.
$$
\end{theorem}

The result of \cite{Leclerc} and the present work concern complementary
regimes.  Leclerc \cite{Leclerc} studies the density-of-states measure, and hence a
phase average of the diagonal return amplitude; the decay exponent is
positive but is not claimed to be sharp.  Here we treat the local
spectral measure at every phase and every lattice site.  In the
Diophantine subcritical regime, we obtain the sharp decay rate
$t^{-1/2}$, and the
exact limit $(\log T)/T$ of the time-averaged return probability.
The mechanism is also different: the decay comes from the
reflectionless Green function and the square-root geometry of the
spectral gaps, rather than from nonlinear self-similarity of the
support \cite{Leclerc}.

We should mention the  arithmetic restriction is necessary for previous results. 
  Let $q_n$ denote the denominators of the continued-fraction
approximants of $\alpha$, and set
$
 \beta(\alpha)
 =\limsup_{n\to\infty}\frac{\log q_{n+1}}{q_n},$ which measures how well the frequency can be approximated by rationals.
The following consequence of Avila-Last-Shamis-Zhou
\cite{AvilaLastShamisZhou} shows that sufficiently Liouville
frequencies lead to a completely different quantitative behavior.

\begin{theorem}\label{prop:liouville-obstruction}
For AMO with $e^{-2/3\beta(\alpha)}<|\lambda|<e^{2/3\beta(\alpha)}$.  Then there exists a dense
$G_\delta$ set $\Theta\subset\mathbb T$ such that, for every
$\theta\in\Theta$ and every $s>0$,
\begin{equation}\label{eq:liop}
 \sup_{0<|I|\le1}\frac{\nu_\theta(I)}{|I|^s} =\infty, \qquad 
 \limsup_{|t|\to\infty}|t|^s|\widehat\nu_\theta(t)|=\infty,
\end{equation}and
\begin{equation}\label{eq:lioa}
 \limsup_{T\to\infty} T^s\frac1T\int_0^T|\widehat\nu_\theta(t)|^2\,dt=\infty.
\end{equation}
\end{theorem}

Theorem~\ref{prop:liouville-obstruction} crosses the transition between
two spectral types.  For $|\lambda|<1$, the spectrum is purely
absolutely continuous \cite{AvilaAMO}; hence
$\widehat\nu_\theta(t)\to0$, but \eqref{eq:liop} shows that the decay
has no positive power-law rate.  For
$1\leq|\lambda|<e^{2\beta(\alpha)/3}<e^{\beta(\alpha)},$
the spectrum is purely singular continuous for every phase
\cite{AYZ1}.  Wiener's theorem still gives decay of the time-averaged
Fourier correlation, while \eqref{eq:lioa} excludes every positive
power-law rate.  Thus the obstruction persists across a change of
spectral type: absolute or singular continuity determines the
qualitative nature of the measure, whereas the arithmetic of the
frequency controls its quantitative temporal decay.  In the singular
continuous regime, \eqref{eq:liop} excludes polynomial pointwise
decay but does not assert that $\widehat\nu_\theta(t)\to0$.

\subsection{Square-root envelopes at spectral edges}

The main  difficulty in Problem~\ref{prob:AJ} and in sharp Fourier
decay is the accumulation of square-root singularities at spectral
gap edges \cite{AJ,ACO}.   Indeed, after the completion of the exact local
distribution  for the subcritical almost Mathieu operator
\cite{LiYouZhou}, Avila \cite{Aqi} asked the second author whether these edge
singularities admit a global pointwise description.  We formulate his
question as follows.

\begin{problem} \cite{Aqi}\label{prob:Avila}
For the subcritical almost Mathieu operator with Diophantine
frequency, determine whether its spectral density admits summable
upper and lower square-root envelopes over the spectral gap edges.
\end{problem}

Now we write the spectrum of AMO as
$ \Sigma=[b_0,a_0]\setminus
 \bigcup_{k\in\mathbb{Z}\backslash\{0\}}(a_k,b_k),$
where $(a_k,b_k)$  is the  open gap by gap labelling theorem \cite{JM}, and all the gaps are open by solution to non-critical Dry Ten Martini Problem \cite{AJ,AJAR,AYZ}. 
The index $0$ records the two exterior edges.  Since $\nu_\theta$ is
purely absolutely continuous \cite{AvilaAMO}, write
$ d\nu_\theta(E)=w_\theta(E)\,dE.$
For $k\in\mathbb{Z}$, let
$$
 U_{k,\theta,-}\subset\Sigma\cap(-\infty,a_k),
 \qquad
 U_{k,\theta,+}\subset\Sigma\cap(b_k,\infty).
$$
We say that these sets have \emph{one-sided density one at
the spectral edges} if for any $k\in\mathbb{Z}$
$$
 \lim_{r\downarrow0} \frac{|U_{k,\theta,-}\cap(a_k-r,a_k)|}{r}=
 \lim_{r\downarrow0} \frac{|U_{k,\theta,+}\cap(b_k,b_k+r)|}{r}=1.
$$
The following theorem gives the upper envelope requested by Avila and
identifies the correct form of the lower envelope.

\begin{theorem}\label{thm:qp-density}
Consider the almost Mathieu operator with $0<\lambda<1$ and
Diophantine frequency $\alpha$.  There exist positive coefficients
$(c_k)_{k\in\mathbb{Z}}$ and $(d_k)_{k\in\mathbb{Z}}$ satisfying
$ \sum_{k\in\mathbb{Z}}(c_k+d_k)<\infty,$
and measurable sets $U_{k,\theta,\pm}$ having one-sided density one
at the spectral edges, such
that, for Lebesgue almost every
$E\in\mathbb R$,
\begin{equation}\label{eq:upp-bound}
     0\le w_\theta(E) \le \mathbf1_\Sigma(E)
 \sum_{k\in\mathbb{Z}}c_k \left( \frac1{\sqrt{|E-a_k|}} +\frac1{\sqrt{|E-b_k|}} \right),
\end{equation}
whereas
\begin{equation}\label{eq:low-bound}
 w_\theta(E) \ge \frac1{4\pi}\mathbf1_\Sigma(E)
 +\mathbf1_\Sigma(E) \sum_{k\in\mathbb{Z}}d_k \left(
 \frac{\mathbf1_{U_{k,\theta,-}}(E)}{\sqrt{a_k-E}}
 +\frac{\mathbf1_{U_{k,\theta,+}}(E)}{\sqrt{E-b_k}}
 \right).
\end{equation}
The coefficients are independent of $\theta$.
\end{theorem}

The upper bound is a genuine almost-everywhere envelope on the whole
spectrum.  The density-one qualification cannot
in general be removed, even for the almost Mathieu operator.

\begin{proposition}\label{prop:failure}
Consider the almost Mathieu operator with $0<\lambda<1$ and
Diophantine frequency $\alpha$.  At the phase $\theta=0$, every open
gap $(a_k,b_k)$, $k\in\mathbb{Z}\backslash\{0\}$, satisfies
\begin{equation}\label{eq:extra}
\begin{aligned}
\lim_{r\downarrow0}
\operatorname*{ess\,inf}_{E\in\Sigma\cap(a_k-r,a_k)}
w_0(E)\left({a_k-E}\right)^{1/2}=0,\quad
\lim_{r\downarrow0}\operatorname*{ess\,inf}_{E\in\Sigma\cap(b_k,b_k+r)}
w_0(E)\left({E-b_k}\right)^{1/2}=0.
\end{aligned}
\end{equation}
\end{proposition}

Thus the density exhibits the expected square-root profile on an
asymptotically full part of every one-sided edge neighborhood, but no
positive square-root lower bound can hold essentially uniformly on
the entire neighborhood. 
\begin{remark}
     We state Theorem~\ref{thm:qp-density} and Proposition \ref{prop:failure} for
the almost Mathieu operator because this is the setting of Avila's
question. Both the upper \eqref{eq:upp-bound} and lower envelopes \eqref{eq:low-bound} hold throughout the subcritical Diophantine regime,  while the evenness of the potential and the Cantor structure of the spectrum are extra input for \eqref{eq:extra}.
\end{remark}

\subsection{Reflectionless measures and spectral geometry}

Classical approaches to spectral continuity relate the growth of
solutions and transfer matrices to Weyl functions and subordinacy
theory \cite{ACO,AJ,GilbertPearson,JitomirskayaLast,JL2,Teschl,LiYouZhou}.  Our
viewpoint is different: we ask how the geometry of the spectrum
controls the concentration and Fourier decay of spectral measures.
We develop this mechanism within the class of reflectionless measures,
where a suitable Craig-type condition on the spectral gaps yields the
upper envelope and Fourier estimates directly, without recourse to
subordinacy theory.

For a finite positive measure $\mu$, let
$$
 F_\mu(z)=\int_{\mathbb R}\frac{d\mu(E)}{E-z},
 \qquad z\in\mathbb C^+,
$$
be its Borel transform.  We call $\mu$ reflectionless on a measurable
set $S\subset\mathbb R$ if
$\operatorname{Re}F_\mu(E+i0)=0$ for Lebesgue almost every $E\in S$.
For compact $S$, let $\mathcal R(S)$ denote the probability measures
supported on $S$ and reflectionless on $S$.

Reflectionless operators first appeared in one-dimensional scattering
and inverse spectral theory \cite{DeiftTrubowitz,Marchenko,Teschl}.
Kotani theory identifies reflectionless behavior on
zero-Lyapunov-exponent sets in ergodic families \cite{Kotani}, and
Remling's theorem identifies reflectionless right limits on the
absolutely continuous spectrum \cite{Remling}.  Sodin and Yuditskii
developed the inverse theory of almost periodic Jacobi matrices with
homogeneous spectrum \cite{SodinYuditskii}; see also \cite{GKT,Teschl}.
Absolute continuity under homogeneous and weakly homogeneous
hypotheses was studied further in
\cite{GesztesyYuditskii,PoltoratskiRemling}.  For the example of reflectionless measures with nonempty singular components, see \cite{NVY,Li}.
Reflectionless measures
also arise in potential theory: the equilibrium measure of a compact
set of positive capacity, or equivalently its harmonic measure at
infinity, is reflectionless \cite{PoltoratskiRemling}.

The relevance of gap geometry is especially clear in inverse spectral
theory.  For a fixed spectrum, Dirichlet data, together with their sign
or phase information, parametrize the remaining freedom in a
reflectionless operator.  Translation and nonlinear evolution induce
flows on these data.  With infinitely many gaps, one must control the
interaction of the corresponding coordinates.  Craig-type conditions
do so through gap lengths, gap separations, and products over all
other gaps \cite{Craig}.  This mechanism underlies the work of Binder,
Damanik, Goldstein, and Luki\'c on KdV \cite{BDGL}, which established
Deift's conjecture for a class containing small analytic
quasiperiodic data with Diophantine frequencies
\cite{DeiftProblems}.  Related results concern the Toda lattice and
uniqueness for the KdV hierarchy \cite{BDLV,LukicYoung}; Fillman, Li,
Luki\'c, and Zhou developed an analogous theory for the cubic
defocusing nonlinear Schr\"odinger equation \cite{FLLZ}, while Li and Luki\'c recently established the corresponding theory for the defocusing mKdV equation \cite{LL}.  These
theorems require quantitative spectral hypotheses, and the
unrestricted KdV conjecture was disproved by Chapouto, Killip, and
Vi\c{s}an \cite{CKV}.

Our use of gap geometry is different.  The inverse-spectral evolution
theory controls the motion of Dirichlet data on a fixed spectrum.  We
instead fix a reflectionless measure and vary the spectral parameter,
uniformly over the full family $\mathcal R(\Sigma)$.  Products of the
same general type now control the combined effect of infinitely many
square-root edge singularities.

Let
$$
\Sigma=[b_0,a_0]\setminus\bigcup_{j\in\mathcal K}G_j,
\qquad
G_j=(a_j,b_j),
\qquad
\gamma_j=b_j-a_j,\qquad \gamma_0=a_0-b_0
$$
where $|\Sigma|>0$, $\mathcal K\subset\mathbb Z\setminus\{0\}$ labels the open
gaps.   Let $\mathcal K_0=\mathcal{K}\cup\{0\}$ and
$$\eta_{j l}=\operatorname{dist}(G_j,G_ l),\quad \eta_{0,j}=\eta_{j,0}=\min\{a_j-b_0, a_0-b_j\}$$
be the distance of different gap.
The quantity
$$
\mathcal C_j=\prod_{\substack{ l\in\mathcal K_0\\  l\neq j}}
\left(1+\frac{\gamma_ l}{\eta_{j l}}\right)^{1/2}
$$
is the convenient version of the classical Craig interaction factor associated with $G_j$ appearing in infinite-gap inverse spectral theory and in
the Craig-type conditions
\cite{BDGL,BDLV,Craig,LukicYoung}. 
The geometric hypothesis used in this paper is the following
\emph{summable square-root Craig-type condition}:
\begin{equation}\label{eq:sqrt-craig}
 \sum_{j\in\mathcal K_0}
 \sqrt{\gamma_j}\,\mathcal C_j<\infty.
\end{equation}
It is adapted to the square-root singularities of reflectionless
densities.  Unlike the classical Craig conditions for Dubrovin flows,
\eqref{eq:sqrt-craig} controls variation in the spectral variable.

The exact partial-fraction expansion of the squared Borel transform is the basic structural identity behind the results below. It also gives the corresponding exact formula for the  density.
\begin{theorem}\label{thm:end}
Assume that \eqref{eq:sqrt-craig} holds. For every $\mu\in\mathcal R(\Sigma)$, there exist coefficients $0\leq \beta_{k,\pm}
\leq
C_*\sqrt{\gamma_k}\mathcal C_k$, $k\in\mathcal K_0$, such that
\begin{equation}\label{pf}
F_\mu(z)^2=\sum_{k\in\mathcal K_0}\bigg(\frac{\beta_{k,-}^2}{z-a_k}-\frac{\beta_{k,+}^2}{z-b_k}\bigg),\qquad z\in\mathbb C^+.
\end{equation}
The series converges absolutely and locally uniformly. Moreover, $\mu$ is absolutely continuous, and, writing $d\mu(E)=w_\mu(E)\,dE$,
\begin{equation}\label{dens}
\pi^2w_\mu(E)^2=\sum_{k\in\mathcal K_0}\bigg(\frac{\beta_{k,-}^2}{a_k-E}+\frac{\beta_{k,+}^2}{E-b_k}\bigg)
\end{equation}
for Lebesgue almost every $E\in\Sigma$. 
\end{theorem}

Our hypothesis \eqref{eq:sqrt-craig} is a  variant of the classical Craig-type
conditions, built from the gap-interaction factors used in inverse
spectral theory \cite{Craig,BDGL,BDLV}.
It also has a geometric consequence: it implies weak homogeneity
of $\Sigma$ (Lemma~\ref{lem:weak-hom}).
Thus every reflectionless measure supported on $\Sigma$ is
absolutely continuous by Poltoratski-Remling \cite{PoltoratskiRemling}.
Our proof establishes absolute continuity independently, together
with quantitative density and Fourier estimates, through the
endpoint residue expansion ~\eqref{pf}.
 
Formula~\eqref{pf} is the main structural identity of the paper.
It replaces the canonical infinite product of reflectionless measure \eqref{prod}     by an additive expansion
over the spectral edges.  Here and below, $\beta_{k,\pm}=\beta_{k,\pm}(\mu)$ denotes the
residue associated with $\mu$, the same residues $\beta_{k,\pm}$ determine
the edge behavior of the density and control its Poisson
regularization.   
The coefficient bound in \eqref{dens} immediately turns the exact identity into a summable square-root envelope and yields the measurable-set estimate, as shown in Corollary \ref{thm:holder}. 
Moreover, as a direct consequence of \eqref{dens}, we have the following:

\begin{corollary}\label{cor:L1}
  Assume that \eqref{eq:sqrt-craig} holds.  For every fixed $k\in\mathcal K_0$, 
    \begin{equation}\label{L1}
\lim_{r\downarrow0}\sup_{\mu\in\mathcal R(\Sigma)}\|\sqrt{r}w_\mu(a_k-rs)-\frac{\beta_{k,-}}{\pi\sqrt s}\|_{L_1(0,1)}=0.
\end{equation}
The analogous assertion holds at $b_k$. 
\end{corollary}

Corollary~\ref{cor:L1} is a
uniform $L^1$ blow-up theorem: after rescaling an $r$-neighborhood
of $a_k$ to the fixed interval $(0,1)$, the density converges to the
square-root profile
$\frac{\beta_{k,-}(\mu)}{\pi\sqrt{s}},
$
uniformly over the reflectionless family $\mathcal R(\Sigma)$. 
The conclusion is  an $L^1$ statement rather than a
pointwise one.  It allows exceptional sets accumulating at the edge,
as is unavoidable for spectra with infinitely many gaps (as shown in  Proposition \ref{prop:failure}), while still
determining the complete tangent profile and the exact leading
asymptotics of the spectral mass.  Moreover, if 
$\beta_{k,-}(\mu)>0,$
then one  turns \eqref{L1} into a square-root lower
envelope on a set of density one at $a_k$, as shown in Corollary \ref{prop:density-one1}.

Another consequence of Theorem~\ref{thm:end} concerns Fourier decay.
For absolutely continuous measures, classical estimates obtain decay rates from
regularity of the density, through 
oscillatory integral estimates \cite{Stein}. Power-law singularities
require separate analysis and can determine the leading Fourier
asymptotics \cite{ACOPower}. Without additional assumptions,
absolute continuity is compatible with the failure of every
positive-power decay rate \cite{ACOSlow}.
For singular continuous measures, even convergence to zero is an
additional property, usually called the Rajchman property.
Classical constructions of Salem and Kaufman provide singular
measures with power-law Fourier decay, revealing a close connection
between Fourier decay and the geometry and arithmetic of the
support \cite{Salem,Kaufman,Mattila}.
More recent work establishes decay for measures defined by
dynamical systems. Jordan and Sahlsten proved polynomial decay
for a class of Gibbs measures for the Gauss map, using oscillatory
estimates and large-deviation control \cite{JordanSahlsten}.
Bourgain and Dyatlov obtained polynomial decay for
Patterson-Sullivan measures by combining nonlinearity with
discretized sum-product estimates \cite{BourgainDyatlov}.
Li established Fourier decay for Furstenberg measures associated
with random walks on split semisimple Lie groups \cite{LiFourier}.
These results extract cancellation from arithmetic or dynamical
structure, without requiring a smooth density.

Pointwise decay must also be distinguished from decay after time
averaging. Wiener's theorem detects atoms through averaged squared
Fourier transforms, while the estimates of Strichartz and Last
relate quantitative time averages to the concentration of the
measure \cite{Wiener,Strichartz,Last}.
Our argument uses the residue identity \eqref{pf} to address both
questions: it controls the error of Poisson smoothing and hence
pointwise Fourier decay, while the endpoint residues determine
the leading term of the time-averaged squared Fourier transform.

\begin{theorem}\label{thm:fourier}
Assume that \eqref{eq:sqrt-craig} holds and let $\mu\in\mathcal R(\Sigma)$. Then
\begin{equation}\label{Fdecay}
 |\widehat\mu(t)|\le \frac{16}{\pi(1-e^{-1})}
\bigg(\sum_{k\in\mathcal K_0}
(\beta_{k,-}+\beta_{k,+})\bigg)
|t|^{-1/2},\quad |t|\geq 1,
\end{equation}
\begin{equation}\label{sharp}
    \limsup_{|t|\to\infty}|t|^{1/2}|\widehat\mu(t)| \ge\frac1{\sqrt\pi}\bigg(\sum_{k\in\mathcal K_0}\bigl(\beta_{k,-}^2+\beta_{k,+}^2\bigr)\bigg)^{1/2},
\end{equation} 

and$$
\lim_{T\to\infty}\frac{1}{\log T}\int_0^T|\widehat\mu(t)|^2\,dt=\frac1\pi\sum_{k\in\mathcal K_0}\bigl(\beta_{k,-}^2+\beta_{k,+}^2\bigr).
$$
\end{theorem}

Theorem \ref{thm:fourier} shows  the endpoint residues $\beta_{k,\pm}$ also  determines the leading coefficient of the time-averaged return probability. This reveals the fact that local edge data can determine the long-time average.

\subsection{Difficulties and ideas of the proofs}

As we mentioned, in the quasiperiodic setting, classical approaches   combine dynamical estimates (mainly from almost reducibility of the Schr\"odinger cocycle) with subordinacy
theory to control spectral concentration
\cite{JitomirskayaLast,JL2,AJ,LiYouZhou}.
Fourier estimates require further control of variation in energy.
The integrated density of states is encoded by the rotation number,
a conjugacy invariant up to an explicit degree correction, and can
be tracked through almost reducibility \cite{BambusiZhao,Z,ZZ}.
Local spectral measures are not conjugacy invariants: even after
reduction to a constant rotation, their densities depend on the
conjugating matrix \cite{GK23}.
Thus approximating the reduced dynamics does not automatically
control the energy dependence of the original spectral weights.
We use cocycle theory (quantitative almost
reducibility) to verify the gap geometry \cite{AJAR,LYZZ}, and
reflectionlessness and the residue identity to control the
densities and their Fourier transforms directly.

The structural simplification is to square the reflectionless
product. Lemma~\ref{lem:fin} then replaces multiplicative gap
interactions \eqref{prod} by an additive expansion over the endpoints \eqref{pf}.
The square-root Craig-type condition \eqref{eq:sqrt-craig} controls the square roots
of the absolute residues, allowing this expansion to survive
infinitely many gaps in Theorem~\ref{thm:end}.
Discarding the signs in \eqref{dens} gives the upper envelope
of Corollary~\ref{thm:holder}; isolating one pole gives the
uniform $L^1$ profile of Corollary~\ref{cor:L1}.
The latter identifies the leading singularity without excluding
cancellation on smaller exceptional sets.

A positive lower profile requires more: an endpoint residue $\beta_{k,\pm}$
may vanish for a general reflectionless measure (as the Dirichlet date might lies in the gap edges).
For the canonical Schr\"odinger measure, the density floor
$w_\theta(E)\geq\frac1{2\pi}$ \cite{AvilaAMO} and weakly homogeneous (Lemma \ref{lem:weak-hom}) force the Borel transform to
diverge with opposite signs at the gap edges.
Its zero stays strictly inside the gap, giving phase-uniform
positivity of the endpoint coefficients for each fixed gap, and
the integrated profile then yields the density-one lower bound
of Corollary~\ref{prop:density-one1}.

The signed expansion (Theorem~\ref{thm:end}) also suggests the obstruction to a full
one-sided lower bound. Dirichlet zeros are exact cancellation
points, but lie outside the spectrum. The key is to transfer
this balance to a spectral set of positive measure.
For even potentials, the symmetry 
turns a canonical Dirichlet zero into a point where the
single-site Green function is controlled above and away from
zero. Analytic distortion and homogeneity of the spectrum \cite{LYZZ} convert this control
into a positive-measure set of bounded density nearby.
Accumulating gaps produce such sets at every edge scale
(Proposition~\ref{cancel}).
In \eqref{dens}, the remaining poles must therefore cancel
the fixed edge contribution on these sets.
Their relative measure tends to zero, explaining their
compatibility with the density-one profile.

For pointwise Fourier decay, Theorem~\ref{thm:fourier}
uses Poisson regularization rather than differentiation across
accumulating singularities.
Writing $u_y=\pi^{-1}\operatorname{Im}F_\mu(\,\cdot+iy)$,
the decisive comparison is
$$
 \pi^2(w_\mu(x)-u_y(x))_+^2 \le \left|
 \pi^2w_\mu(x)^2+ \operatorname{Re}\bigl(F_\mu(x+iy)^2\bigr) \right|.
$$
Positivity makes the uncontrolled real part enter with a
favorable sign. The residue identities now reduce the error
to explicit pole kernels; conservation of mass turns the
one-sided comparison into the full $L^1$ estimate.
This supplies the fractional regularity needed for
inverse-square-root Fourier decay \cite{Stein}.
It is information furnished by the identity \eqref{pf}, not by the
upper envelope alone.

For the exact average, the crucial observation is instead
$|F_\mu|^2=|F_\mu^2|$: the quadratic quantity in Plancherel's
theorem can be read from the additive pole expansion.
Each endpoint contributes a logarithmic divergence, and
summability controls the contribution of the accumulating
tail. Karamata’s Tauberian theorem yields the exact asymptotic
in Theorem~\ref{thm:fourier}.
Its positive coefficient proves optimality of the one-half
exponent. Thus the familiar connection between density
singularities and Fourier asymptotics \cite{ACOPower}
extends here to infinitely many interacting edges, with
the residues determining the leading coefficient.

\section{Endpoint residue representation of reflectionless measures}

\subsection{Reflectionless measures and finite-gap approximation}
\label{subsec:reflectionless}

We first recall basis of reflectionless measures. 
Let $\Sigma\subset\mathbb R$ be a compact set of positive
Lebesgue measure, with
$b_0=\min\Sigma,
a_0=\max\Sigma.$
Write
$$
\Sigma=[b_0,a_0]\setminus\bigcup_{k\in\mathcal K}G_k,
\qquad
G_k=(a_k,b_k),
\qquad
\gamma_k=b_k-a_k,
$$
where the $G_k$ are the bounded components of
$\mathbb R\setminus\Sigma$. The index set $\mathcal K$ is at most countable and need not coincide with $\mathbb N^+$. Thus, even in the Schr\"odinger setting, we do not assume that all labeled spectral gaps are open.

For a probability measure $\mu$ supported on $\Sigma$, define
its Cauchy transform by
$$
g_\mu(z)=\int_{\mathbb R}\frac{d\mu(E)}{E-z},
\qquad z\in\mathbb C^+.
$$
This is a Herglotz function, normalized by
$g_\mu(z)=-z^{-1}+O(z^{-2})$ at infinity.
We call $\mu$ \emph{reflectionless on $\Sigma$} if
$$
\operatorname{Re}g_\mu(E+i0)=0
\quad\text{for Lebesgue almost every }E\in\Sigma,
$$
and denote the class of such probability measures by
$\mathcal R(\Sigma)$.
Here $g_\mu(E+i0)$ denotes the nontangential boundary value.
We use the Lebesgue almost-everywhere convention throughout;
see \cite{PoltoratskiRemling}.

The exponential representation of Herglotz functions gives
the standard parametrization of $\mathcal R(\Sigma)$ by
Dirichlet data
$\boldsymbol\xi=(\xi_k)_{k\in\mathcal K},
\xi_k\in[a_k,b_k]:$
\begin{equation}
\label{prod}
g_\mu(z)=
-\frac{1}{\sqrt{z-b_0}\sqrt{z-a_0}}\prod_{k\in\mathcal K}
\frac{z-\xi_k}{\sqrt{z-a_k}\sqrt{z-b_k}},
\qquad z\in\mathbb C^+.
\end{equation}
Each square root is the principal branch.
This representation is standard in the inverse spectral
theory of reflectionless operators; see
\cite{SodinYuditskii}.

To recall why the parameters arise, $g_\mu$ is real and
strictly increasing on each gap:
$$
g_\mu'(E)=\int_\Sigma\frac{d\mu(x)}{(x-E)^2}>0,
\qquad E\in G_k.
$$
Thus it has at most one zero in $G_k$, which is $\xi_k$
when such a zero exists.
If $g_\mu$ is positive throughout $G_k$, set $\xi_k=a_k$;
if it is negative throughout $G_k$, set $\xi_k=b_k$.
The boundary argument of $g_\mu$ equals $\pi/2$ almost
everywhere on $\Sigma$, and equals $\pi$ and $0$ on
$(a_k,\xi_k)$ and $(\xi_k,b_k)$, respectively.
These boundary values determine the exponential
representation, while the normalization at infinity fixes
the constant in \eqref{prod}.

We will approximate \eqref{prod}
by finite products.
Assume, for notation, that $\mathcal K$ is infinite and
enumerate it as $\kappa_1,\kappa_2,\ldots$.
For $N\ge0$, set
$\Sigma_N=[b_0,a_0]\setminus\bigcup_{j=1}^N G_{\kappa_j}$
and
\begin{equation*}
g_N(z)=-\frac{1}{\sqrt{z-b_0}\sqrt{z-a_0}}
\prod_{j=1}^N\frac{z-\xi_{\kappa_j}}
{\sqrt{z-a_{\kappa_j}}\sqrt{z-b_{\kappa_j}}}.
\end{equation*}
The empty product is understood to be one.
Each $g_N$ is the Cauchy transform of a unique probability
measure $\mu_N$ supported on $\Sigma_N$. Under the gap-separation condition,
its boundary values give
$d\mu_N(E)=w_N(E)\,dE,$
where
\begin{equation*}
w_N(E)=\frac{\mathbf1_{\Sigma_N}(E)}
{\pi\sqrt{(E-b_0)(a_0-E)}}\prod_{j=1}^N\frac{|E-\xi_{\kappa_j}|}
{\sqrt{|E-a_{\kappa_j}|\,|E-b_{\kappa_j}|}}
\end{equation*}
almost everywhere.
In particular, every finite-gap approximant is absolutely
continuous.

Moreover,
$g_N\longrightarrow g_\mu$ locally uniformly in $\mathbb C^+$, and consequently 
$\mu_N\rightharpoonup\mu.$
Weak convergence alone does not preserve absolute
continuity \cite{NVY,Li}. Thus no absolute continuity of $\mu$ is assumed
at this stage; it will follow from Craig-type conditions established below.

 \subsection{Endpoint residue representation}

For any $\mu\in\mathcal R(\Sigma)$, our main observation is that, 
for finite-gap approximants, squaring the canonical product  \eqref{prod} 
produces a rational function whose only possible poles are
simple poles at the spectral endpoints.  Under the summability
assumption, the corresponding partial-fraction expansions
converge to an exact endpoint residue representation for
$F_\mu^2$, with no additional analytic term.

The proof is based on the following finite-gap identity. Its uniform form will also be used in the Fourier estimate.

\begin{lemma}\label{lem:fin}
Let $J\subset\mathcal K$ be finite and set
$$
g_J(z)=-\frac{1}{\sqrt{z-b_0}\sqrt{z-a_0}}\prod_{k\in J}\frac{z-\xi_k}{\sqrt{z-a_k}\sqrt{z-b_k}}.
$$
There exist nonnegative coefficients $\beta_{k,\pm}^J$, $k\in J\cup\{0\}$, such that
\begin{equation}\label{pfJ}
g_J(z)^2=\sum_{k\in J\cup\{0\}}\bigg(\frac{(\beta_{k,-}^J)^2}{z-a_k}-\frac{(\beta_{k,+}^J)^2}{z-b_k}\bigg),
\end{equation}
and
\begin{equation}\label{genv}
|g_J(z)|\leq C_*\sum_{k\in J\cup\{0\}}\sqrt{\gamma_k}\mathcal C_k\bigg(\frac{1}{\sqrt{|z-a_k|}}+\frac{1}{\sqrt{|z-b_k|}}\bigg).
\end{equation}
The constants are independent of $J$ and of the Dirichlet data.
\end{lemma}

\begin{proof}
Squaring the product gives
$$
g_J(z)^2=\frac{1}{(z-b_0)(z-a_0)}\prod_{k\in J}\frac{(z-\xi_k)^2}{(z-a_k)(z-b_k)}.
$$
This is a rational function whose poles are simple and contained in the endpoints $a_k,b_k$, $k\in J\cup\{0\}$. Moreover,
$
g_J(z)^2=O(|z|^{-2}),\quad  z\to\infty.
$
Hence its partial fraction decomposition is completely determined by the residues.

Fix $k\in J$. At the left endpoint $a_k$,
$$
\begin{aligned}
\operatorname*{Res}_{z=a_k}g_J(z)^2
=\frac{(a_k-\xi_k)^2}{(a_k-b_0)(a_k-a_0)(a_k-b_k)}
\prod_{\substack{l\in J\\l\ne k}}\frac{(a_k-\xi_l)^2}{(a_k-a_l)(a_k-b_l)}
\end{aligned}
$$
We therefore set
$$
\beta_{k,-}^J=
\frac{\xi_k-a_k}{\sqrt{\gamma_k}\sqrt{(a_k-b_0)(a_0-a_k)}}
\prod_{\substack{l\in J\\l\ne k}}
\frac{|a_k-\xi_l|}{\sqrt{|a_k-a_l||a_k-b_l|}}\geq 0.
$$
Then
$
\operatorname*{Res}_{z=a_k}g_J(z)^2=(\beta_{k,-}^J)^2.
$

Similarly, at the right endpoint $b_k$, if we set 
$$
\beta_{k,+}^J=\frac{b_k-\xi_k}{\sqrt{\gamma_k}\sqrt{(b_k-b_0)(a_0-b_k)}}
\prod_{\substack{l\in J\\l\ne k}}
\frac{|b_k-\xi_l|}{\sqrt{|b_k-a_l||b_k-b_l|}}\geq 0,
$$
then 
$
\operatorname*{Res}_{z=b_k}g_J(z)^2=-(\beta_{k,+}^J)^2.
$

Moreover, if we set
$$
\beta_{0,-}^J=\frac{1}{\sqrt{\gamma_0}}\prod_{l\in J}\frac{a_0-\xi_l}{\sqrt{(a_0-a_l)(a_0-b_l)}},
\qquad
\beta_{0,+}^J=\frac{1}{\sqrt{\gamma_0}}\prod_{l\in J}\frac{\xi_l-b_0}{\sqrt{(a_l-b_0)(b_l-b_0)}},
$$
it is immediate to check that $\operatorname*{Res}_{z=a_0}g_J(z)^2=(\beta_{0,-}^J)^2$, $
\operatorname*{Res}_{z=b_0}g_J(z)^2=-(\beta_{0,+}^J)^2,
$
which proves \eqref{pfJ}.

For $k\in J$,
$
\xi_k-a_k\leq\gamma_k, b_k-\xi_k\leq\gamma_k.
$
If $l\ne k$, then for either endpoint $e\in\{a_k,b_k\}$,
\begin{equation}\label{eq:one-gap}
  \bigg(1+\frac{\gamma_l}{\eta_{kl}}\bigg)^{-1/2} \leq \frac{|e-\xi_l|}{\sqrt{|e-a_l||e-b_l|}}
\leq\bigg(1+\frac{\gamma_l}{\eta_{kl}}\bigg)^{1/2}.
\end{equation}
Moreover, if $x=a_k-b_0$ and $y=a_0-a_k$, then $x+y=\gamma_0$ and $\min\{x,y\}\geq\eta_{k0}$. Therefore
$$
\frac{1}{\sqrt{(a_k-b_0)(a_0-a_k)}}
\leq\frac{\sqrt2}{\gamma_0}\bigg(1+\frac{\gamma_0}{\eta_{k0}}\bigg)^{1/2}.
$$
The same estimate holds with $a_k$ replaced by $b_k$. It follows from the definition of $\mathcal C_k$ that
$
\beta_{k,\pm}^J\leq\frac{\sqrt{2\gamma_k}}{\gamma_0}\mathcal C_k,k\in J.
$
For $k=0$, the same one-gap estimate gives
$
\beta_{0,\pm}^J\leq\frac{\mathcal C_0}{\sqrt{\gamma_0}}.
$
After enlarging $C_*$, this proves 
\begin{equation}\label{bbound}
\beta_{k,\pm}^J\leq C_*\sqrt{\gamma_k}\mathcal C_k,\qquad k\in J\cup\{0\}.
\end{equation} 
Finally, \eqref{pfJ} gives
$$
|g_J(z)|^2\leq\sum_{k\in J\cup\{0\}}\bigg(\frac{(\beta_{k,-}^J)^2}{|z-a_k|}+\frac{(\beta_{k,+}^J)^2}{|z-b_k|}\bigg).
$$
Using $\sqrt{x_1^2+\cdots+x_m^2}\leq x_1+\cdots+x_m$ for nonnegative $x_j$, together with \eqref{bbound}, proves \eqref{genv}.
\end{proof}
\begin{proof}[Proof of Theorem~\ref{thm:end}]
Let $J_n\uparrow\mathcal K$ be an exhaustion by finite sets. Since
$$1+\sum_{l\ne k}\frac{\gamma_l}{\eta_{kl}}\leq \Pi_{l\neq k}\bigg(1+\frac{\gamma_l}{\eta_{kl}}\bigg)\leq \mathcal C_k^2,$$
\eqref{eq:sqrt-craig} implies for every fixed $k$, one has
$\sum_{l\ne k}\frac{\gamma_l}{\eta_{kl}}<\infty.$
For every fixed $k\in\mathcal K$ and $l\ne k$, the one-gap estimate \eqref{eq:one-gap} gives
$$
\left|\log\frac{|a_k-\xi_l|}{\sqrt{|a_k-a_l||a_k-b_l|}}\right|
\leq
\frac12\log\bigg(1+\frac{\gamma_l}{\eta_{kl}}\bigg)
\leq
\frac{\gamma_l}{2\eta_{kl}},
$$
uniformly in $\xi_l\in[a_l,b_l]$. Hence the endpoint products converge uniformly. The same argument applies at $b_k$, and for $k=0$. Therefore, along any exhaustion $J_n\uparrow\mathcal K$,
$
\beta_{k,\pm}^{J_n}\longrightarrow\beta_{k,\pm}
$
for every fixed $k\in\mathcal K_0$.

By \eqref{bbound},
$
\beta_{k,\pm}^{J_n}\leq C_*\sqrt{\gamma_k}\mathcal C_k.
$
Passing to the limit gives
\begin{equation}\label{bbound-inf}
\beta_{k,\pm}\leq C_*\sqrt{\gamma_k}\mathcal C_k.
\end{equation}
Condition \eqref{eq:sqrt-craig} implies
$
\sum_{k\in\mathcal K_0}\sqrt{\gamma_k}\mathcal C_k<\infty.
$
Therefore, on every compact subset of $\mathbb C^+$, the series in \eqref{pfJ} is dominated by a summable sequence independent of $n$. Letting $n\to\infty$ and using
$g_{J_n}\longrightarrow F_\mu$
locally uniformly in $\mathbb C^+$ proves \eqref{pf}.

We next obtain the upper envelope and absolute continuity. Let $\mu_n$ be the finite-gap reflectionless measure with Borel transform $g_{J_n}$. Its density is
$$
w_n(E)=\frac{\mathbf1_{\Sigma_n}(E)}
{\pi\sqrt{(E-b_0)(a_0-E)}}
\prod_{k\in J_n}
\frac{|E-\xi_k|}
{\sqrt{|E-a_k||E-b_k|}}.
$$
Since $g_{J_n}\to F_\mu$ locally uniformly in $\mathbb C^+$, one has
$\mu_n\rightharpoonup\mu.$
Taking boundary values in \eqref{genv} gives
\begin{equation}\label{env-boundary}
w_n(E)\leq
C_*\mathbf1_{[b_0,a_0]}(E)
\sum_{k\in\mathcal K_0}
\sqrt{\gamma_k}\mathcal C_k
\bigg(\frac{1}{|E-a_k|^{1/2}}+\frac{1}{|E-b_k|^{1/2}}\bigg)
\end{equation}
for Lebesgue almost every $E\in\mathbb R$, uniformly in $n$ and in the Dirichlet data.

The right-hand side belongs to $L^1(\mathbb R)$. Indeed, for every $e\in[b_0,a_0]$,
$
\int_{b_0}^{a_0}\frac{dE}{|E-e|^{1/2}}\leq2\sqrt{2\gamma_0},
$
while
$
\sum_{k\in\mathcal K_0}\sqrt{\gamma_k}\mathcal C_k<\infty.
$
Hence, for every nonnegative $\varphi\in C_c(\mathbb R)$, weak convergence and \eqref{env-boundary} give
$$
\begin{aligned}
\int_{\mathbb R}\varphi(E)\,d\mu(E)
&=
\lim_{n\to\infty}\int_{\mathbb R}\varphi(E)w_n(E)\,dE\\
&\leq
C_*\int_{\mathbb R}\varphi(E)\mathbf1_{[b_0,a_0]}(E)
\sum_{k\in\mathcal K_0}\sqrt{\gamma_k}\mathcal C_k
\bigg(\frac{1}{|E-a_k|^{1/2}}+\frac{1}{|E-b_k|^{1/2}}\bigg)\,dE.
\end{aligned}
$$
Thus $\mu$ is absolutely continuous. Writing $d\mu(E)=w_\mu(E)\,dE$, we obtain
\begin{equation}\label{eq:enve}
0\leq w_\mu(E)\leq
C_*\mathbf1_{[b_0,a_0]}(E)\sum_{k\in\mathcal K_0}
\sqrt{\gamma_k}\mathcal C_k
\bigg(\frac{1}{|E-a_k|^{1/2}}+\frac{1}{|E-b_k|^{1/2}}\bigg)
\end{equation}
for Lebesgue almost every $E\in\mathbb R$.

It remains to prove the exact boundary identity. For almost every $E\in\Sigma$, reflectionlessness gives
$
F_\mu(E+i0)=i\pi w_\mu(E).
$
Moreover, by \eqref{bbound-inf},
$$
\sum_{k\in\mathcal K_0}
\bigg(\frac{\beta_{k,-}^2}{|E-a_k|}+\frac{\beta_{k,+}^2}{|E-b_k|}\bigg)\leq
C_*\bigg[\sum_{k\in\mathcal K_0}
\sqrt{\gamma_k}\mathcal C_k\bigg(\frac{1}{|E-a_k|^{1/2}}+\frac{1}{|E-b_k|^{1/2}}\bigg)\bigg]^2
<\infty
$$
for Lebesgue almost every $E$, since the square-root series in brackets is finite almost everywhere by its $L^1$ integrability. We may therefore take boundary values term by term in \eqref{pf}. Since
$$
F_\mu(E+i0)^2=-\pi^2w_\mu(E)^2,
$$
we obtain \eqref{dens}. This completes the proof of Theorem~\ref{thm:end}.
\end{proof}

\begin{corollary}\label{thm:holder}
Assume that \eqref{eq:sqrt-craig} holds.  Then for every
$\mu\in\mathcal R(\Sigma)$,
\begin{equation}\label{env-mu}
w_\mu(E)\leq C_*\mathbf1_{[b_0,a_0]}(E)
\sum_{k\in\mathcal K_0}
\sqrt{\gamma_k}\mathcal C_k
\bigg(\frac{1}{|E-a_k|^{1/2}}+\frac{1}{|E-b_k|^{1/2}}\bigg)
\end{equation}
for Lebesgue almost every $E\in\mathbb R$.
Consequently, there exists a constant $C_\Sigma<\infty$, depending
only on $\Sigma$, such that
$\mu(X)\le C_\Sigma |X|^{1/2}$
for every measurable set $X\subset\mathbb R$.
\end{corollary}
\begin{proof}
The  upper envelope \eqref{env-mu} is just \eqref{eq:enve}.
We only need to prove the measurable-set estimate. For every $e\in\mathbb R$ and every measurable set $X\subseteq\mathbb R$, by the layer-cake representation,
$$
\begin{aligned}
\int_X\frac{dE}{|E-e|^{1/2}}&=
\int_0^\infty\left|X\cap
\left\{E:\frac1{|E-e|^{1/2}}>t\right\}\right|\,dt
\leq\int_0^\infty\min\bigl\{|X|,2t^{-2}\bigr\}\,dt\\
&=|X|\int_0^{\sqrt{2/|X|}}dt
+2\int_{\sqrt{2/|X|}}^\infty t^{-2}\,dt=
2\sqrt2\,|X|^{1/2}.
\end{aligned}
$$
Applying this estimate term by term to \eqref{env-mu} yields
$$
\begin{aligned}
\mu(X)&=\int_Xw_\mu(E)\,dE\leq C_*
\sum_{k\in\mathcal K_0}
\sqrt{\gamma_k}\mathcal C_k
\int_X\bigg(\frac{1}{|E-a_k|^{1/2}}+\frac{1}{|E-b_k|^{1/2}}
\bigg)\,dE\\
&\leq C_*\bigg(\sum_{k\in\mathcal K_0}\sqrt{\gamma_k}\mathcal C_k
\bigg)|X|^{1/2}\leq C_*|X|^{1/2}.
\end{aligned}
$$
The constant is independent of $\mu\in\mathcal R(\Sigma)$. This completes the proof of Corollary~\ref{thm:holder}.
\end{proof}

\subsection{Uniform edge $L^1$ blow-up profiles}

Now let's analysis the uniform edge blow-up of  $w_\mu(a_k-\cdot)$. 

\begin{proof}[Proof of Corollary \ref{cor:L1}:]

Note $F_\mu(E)$ is real for $E\in\mathbb R\setminus\Sigma$, \eqref{pf} gives
\begin{equation}\label{pos}
\pi^2w_\mu(E)^2=\bigg[\sum_{k\in\mathcal K_0}
\bigg(\frac{\beta_{k,-}^2}{a_k-E}+\frac{\beta_{k,+}^2}{E-b_k}\bigg)\bigg]_+
\end{equation}
for almost every $E\in\mathbb R$, where $[x]_+=\max\{x,0\}$.

For $x\geq0$ and $y\in\mathbb R$,
$
\big|\sqrt{[x+y]_+}-\sqrt{x}\big|\leq\sqrt{|y|}.
$
Apply this inequality to \eqref{pos} with $E=a_k-s$ and $x=\beta_{k,-}^2/s$. We obtain
$$
\begin{aligned}
\pi\bigg|w_\mu(a_k-s)-\frac{\beta_{k,-}}{\pi\sqrt s}\bigg|
&\leq\sum_{\substack{l\in\mathcal K_0\\l\ne k}}
\frac{\beta_{l,-}}{\sqrt{|a_k-s-a_l|}}+
\sum_{l\in\mathcal K_0}
\frac{\beta_{l,+}}{\sqrt{|a_k-s-b_l|}}.
\end{aligned}
$$
For every fixed endpoint $e\ne a_k$,
$$
\frac{1}{\sqrt r}\int_0^r\frac{ds}{\sqrt{|a_k-s-e|}}\leq2\sqrt2.
$$
And if $r<|a_k-e|/2$, then
$$
\frac{1}{\sqrt r}\int_0^r\frac{ds}{\sqrt{|a_k-s-e|}}
\leq\sqrt{\frac{2r}{|a_k-e|}},
$$
hence the left-hand side tends to zero as $r\downarrow0$. 
By \eqref{bbound},
$\beta_{l,\pm}\leq C_*\sqrt{\gamma_l}\mathcal C_l,$
and $\sum_l\sqrt{\gamma_l}\mathcal C_l<\infty$. Applying dominated convergence to the sum and then making the change of variables proves \eqref{L1}, uniformly over $\mu\in\mathcal R(\Sigma)$.

\end{proof}

Before turning to the lower profile, we record a geometric consequence of \eqref{eq:sqrt-craig}. The summability of the gap interactions prevents the spectrum from becoming too thin at small scales and, in particular, implies
weak homogeneity. This will also be used below to control the Dirichlet data inside the gaps.

\begin{lemma}\label{lem:weak-hom}
Assume \eqref{eq:sqrt-craig}. Then, for every $x\in\Sigma$,
$$
\limsup_{r\downarrow0}\frac{|\Sigma\cap(x-r,x+r)|}{2r}\ge\frac12.
$$
In particular, $\Sigma$ is weakly homogeneous.
\end{lemma}

\begin{proof}
The residue identity in Theorem~\ref{thm:end} and its coefficient
bounds give
\begin{equation}\label{eq:weak-hom-resolvent}
|F_\mu(x+iy)|^2\le
\frac1y\sum_{k\in\mathcal K_0}
\bigl(\beta_{k,-}^2+\beta_{k,+}^2\bigr)
\le \frac{C_\Sigma}{y},
\qquad x\in\mathbb R,\quad y>0,
\end{equation}
uniformly over $\mu\in\mathcal R(\Sigma)$.

Fix $x\in\Sigma$. By the canonical parametrization, there exists
$\mu_x\in\mathcal R(\Sigma)$ with Dirichlet data
$$
\xi_j=
\begin{cases}
a_j,& b_j\le x,\\
b_j,& a_j\ge x.
\end{cases}
$$
Write $F_x=F_{\mu_x}$. For these choices, each gap factor satisfies
$$
\log\left|\frac{x+iy-\xi_j}{\sqrt{x+iy-a_j}\sqrt{x+iy-b_j}}\right|^2=
\int_{G_j}\frac{|E-x|}{(E-x)^2+y^2}\,dE.
$$
On the other hand, $$
\int_{b_0}^{a_0}\frac{|E-x|}{(E-x)^2+y^2}\,dE
=\log\frac{|x+iy-b_0|\,|x+iy-a_0|}{y^2},
$$
therefore, 
\begin{equation}\label{eq:extremal-modulus}
\begin{aligned}
    |F_x(x+iy)|^2&=\frac{1}{|x+iy-b_0|\,|x+iy-a_0|}\exp\left(
\int_{\cup G_j}\frac{|E-x|}{(E-x)^2+y^2}\,dE
\right)\\
&=\frac1{y^2}\exp\left(-\int_\Sigma\frac{|E-x|}{(E-x)^2+y^2}\,dE\right).
\end{aligned}
\end{equation}
Combining \eqref{eq:weak-hom-resolvent} and
\eqref{eq:extremal-modulus}, we obtain
\begin{equation}\label{eq:log-thickness}
\int_\Sigma\frac{|E-x|}{(E-x)^2+y^2}\,dE
\ge \log y^{-1}-\log C_\Sigma.
\end{equation}

Suppose, contrary to the assertion, that there exist
$0<d<1/2$ and $R>0$ such that
$$
M(t):=|\Sigma\cap(x-t,x+t)|\le 2dt,
\qquad 0<t\le R.
$$
Since
$$
M(t)=\int_0^t
\bigl(\mathbf1_\Sigma(x+u)+\mathbf1_\Sigma(x-u)\bigr)\,du,
$$
the function $M$ is absolutely continuous. For $0<y<R$,
integration by parts yields
$$
\begin{aligned}
\int_\Sigma\frac{|E-x|}{(E-x)^2+y^2}\,dE&=\int_0^\infty\frac{t}{t^2+y^2}\,dM(t)\le \frac{M(y)}{y}
     +\int_y^R\frac{dM(t)}{t}
     +\frac{|\Sigma|}{R}\\
&=\frac{M(R)}{R}
  +\int_y^R\frac{M(t)}{t^2}\,dt
  +\frac{|\Sigma|}{R}\\
&\le 2d\log\frac{R}{y}+2d+\frac{|\Sigma|}{R}.
\end{aligned}
$$
Dividing this estimate and \eqref{eq:log-thickness}
by $\log y^{-1}$ and letting $y\downarrow0$ gives
$1\le 2d$, a contradiction.
\end{proof}

It is plausible that for some  $\mu\in\mathcal R(\Sigma)$, $\beta_{k,\pm}=0$, or say the Dirichlet date $\xi_k$ lies in the gap edges. In this case, the associated square-root singularity disappears.  However, if $\beta_{k,\pm}>0$, one obtain the density-one lower profile from the $L^1$ endpoint asymptotics.

\begin{corollary}\label{prop:density-one1}
Fix $\mu\in\mathcal R(\Sigma)$. If there exists $c>0$, such that $w_\mu(E)\geq c>0
$ for almost every $E\in\Sigma$, then $\beta_{k,\pm}>0.$ As a consequence,  for any $k\in\mathcal K_0$,
$$
U_{k,-}=
\bigg\{E<a_k:w_\mu(E)\geq\frac{\beta_{k,-}}{2\pi\sqrt{a_k-E}}\bigg\}
$$
has one-sided density one at $a_k$. The analogous assertion holds at $b_k$.
\end{corollary}
\begin{proof}
Fix $G_k=(a_k,b_k)$ and $0<r<\gamma_k/2$. Define
$$
M(t):=|\Sigma\cap(a_k-t,a_k)|
=\int_0^t\mathbf1_\Sigma(a_k-u)\,du,
\qquad 0\leq t\leq r.
$$
By Lemma \ref{lem:weak-hom},  there
exist $c_k>0$ and $r_j\downarrow0$ such that
$M(r_j)\geq c_kr_j.$ Passing to a subsequence, assume $r_{j+1}<r_j/4$.
For $0<s<\gamma_k/2$, the contribution to
$F_\mu(a_k+s)$ from $E\leq a_k-r$ is nonpositive,
whereas the contribution from $E\geq b_k$ is at most
$2/\gamma_k$, since $\mu$ is a probability measure.
Meanwhile, $w_\mu(E)\geq c$ for almost every $E\in\Sigma$, as a consequence we obtain
$$
\begin{aligned}
F_\mu(a_k+s)\leq\frac2{\gamma_k}-c
\int_{\Sigma\cap(a_k-r,a_k)}
\frac{dE}{s+a_k-E}=\frac2{\gamma_k}
-c\int_0^r\frac{M'(t)}{s+t}\,dt.
\end{aligned}
$$
For $s>0$ sufficiently small, let
$N(s)=\max\{N:r_N>s\}.$  Integration by parts yields
$$
\begin{aligned}
\int_0^r\frac{M'(t)}{s+t}\,dt=\frac{M(r)}{s+r}
+\int_0^r\frac{M(t)}{(s+t)^2}\,dt
\geq\sum_{j=2}^{N(s)} c_k\int_{r_j}^{2r_j}\frac{r_j}{(s+t)^2}\,dt
\geq\frac{c_k}{9}(N(s)-1).
\end{aligned}
$$
Consequently,
$
F_\mu(a_k+s)\rightarrow-\infty
$
as $s\downarrow0$. 
Similarly,
$F_\mu(b_k-s)\to+\infty$.
Since
$$
F_\mu'(x)=\int_\Sigma\frac{d\mu(E)}{(E-x)^2}>0,
\qquad x\in G_k,
$$
there is a unique zero $a_k<\xi_k<b_k$. It follows that $\beta_{k,\pm}>0$.

By Chebyshev's inequality,
$$
\begin{aligned}
\frac{|(a_k-r,a_k)\setminus U_{k,-}|}{r}
&\leq\frac{2\pi}{\beta_{k,-}\sqrt r}
\int_0^r\left|w_\mu(a_k-s)-\frac{\beta_{k,-}}{\pi\sqrt s}
\right|\,ds
\longrightarrow 0
\end{aligned}
$$
as $r\downarrow0$, by Corollary~\ref{cor:L1}.
Thus $U_{k,-}$ has one-sided density one at $a_k$.
The argument at $b_k$ is identical. For the endpoints $a_0,b_0$, the conclusion follows from  Corollary~\ref{cor:L1} by the same argument.

\end{proof}

\section{Fourier decay}
\subsection{Exact averaged Fourier decay}
The endpoint identity also determines the exact leading term of the time-averaged Fourier correlation. We now derive this directly from \eqref{pf} by Plancherel's theorem and Karamata's Tauberian theorem.

\begin{proposition}\label{prop:avg}
Assume that \eqref{eq:sqrt-craig} holds and let $\mu\in\mathcal R(\Sigma)$. Then
\begin{equation}\label{avg}
\lim_{T\to\infty}\frac{1}{\log T}\int_0^T|\widehat\mu(t)|^2\,dt=
\frac{1}{\pi}\sum_{k\in\mathcal K_0}\big(\beta_{k,-}^2+\beta_{k,+}^2\big)>0,
\end{equation}
where $\beta_{k,\pm}$ are the coefficients in Theorem~\ref{thm:end}.
\end{proposition}
\begin{proof}
Set $S_\mu=\sum_{k\in\mathcal K_0}
\bigl(\beta_{k,-}^2+\beta_{k,+}^2\bigr).$
By \eqref{eq:sqrt-craig}, $S_\mu<\infty$.
Moreover, $S_\mu>0$, since otherwise \eqref{pf} would imply
$F_\mu\equiv0$, contradicting $\mu(\mathbb R)=1$.

 For $y>0$, define
$u_y(x)=\frac{1}{\pi}\operatorname{Im}F_\mu(x+iy).$
Since
$$
u_y=P_y*\mu,\qquad P_y(x)=\frac{1}{\pi}\frac{y}{x^2+y^2},
$$
and hence
$\widehat u_y(t)=e^{-y|t|}\widehat\mu(t).$
By Plancherel's theorem,
\begin{equation}\label{pl1}
\begin{aligned}
\int_{\mathbb R}e^{-2y|t|}|\widehat\mu(t)|^2\,dt&=2\pi\int_{\mathbb R}|u_y(x)|^2\,dx=\frac{2}{\pi}\int_{\mathbb R}|\operatorname{Im}F_\mu(x+iy)|^2\,dx\\
&=\frac{1}{\pi}\int_{\mathbb R}\big(|F_\mu(x+iy)|^2-\operatorname{Re}F_\mu(x+iy)^2\big)\,dx.
\end{aligned}
\end{equation}

Fix $y>0$. Apply Cauchy's theorem to $F_\mu^2$ on the contour
consisting of the segment $[-R+iy,R+iy]$ and the semicircle
$
\Gamma_R=\{iy+Re^{it}:0\leq t\leq\pi\}.$
Since $F_\mu(z)^2=O(|z|^{-2})$ at infinity,
$\int_{-R}^{R}F_\mu(x+iy)^2\,dx
=-\int_{\Gamma_R}F_\mu(z)^2\,dz
=O(R^{-1}).$
Letting $R\to\infty$ gives
\begin{equation}\label{Fzero}
\int_{\mathbb R}F_\mu(x+iy)^2\,dx=0.
\end{equation}

Fix $\varepsilon>0$. There exists a finite set $J\subset\mathcal K_0$ such that
\begin{equation}\label{tail}
\sum_{k\in\mathcal K_0\setminus J}\big(\beta_{k,-}^2+\beta_{k,+}^2\big)<\varepsilon.
\end{equation}
We may assume that $0\in J$. For convenience, denote the finite part of the sum by
$$
S_J(z)=\sum_{k\in J}\bigg(\frac{\beta_{k,-}^2}{z-a_k}-\frac{\beta_{k,+}^2}{z-b_k}\bigg).
$$
By the triangular inequality  and Theorem~\ref{thm:end},
$$
\begin{aligned}
\big||F_\mu(x+iy)|^2-|S_J(x+iy)|\big|
&\leq\sum_{k\in\mathcal K_0\setminus J}
\bigg(\frac{\beta_{k,-}^2}{|x+iy-a_k|}+\frac{\beta_{k,+}^2}{|x+iy-b_k|}\bigg).
\end{aligned}
$$

We first integrate on the fixed interval
$I=[b_0-\gamma_0,a_0+\gamma_0].$
For every $e\in[b_0,a_0]$ and $0<y\leq \frac{1}{2}$,
$$
\begin{aligned}
\int_I\frac{dx}{|x+iy-e|}=
\int_{b_0-\gamma_0}^{a_0+\gamma_0}\frac{dx}{\sqrt{(x-e)^2+y^2}}=
\operatorname{arsinh}\frac{a_0+\gamma_0-e}{y}+
\operatorname{arsinh}\frac{e-b_0+\gamma_0}{y}\leq C_*\log y^{-1}
\end{aligned}
$$
 for $0<y\leq\frac{1}{2}$. Therefore, by \eqref{tail},
\begin{equation}\label{tailest}
\int_I\big||F_\mu(x+iy)|^2-|S_J(x+iy)|\big|\,dx
\leq C_*\varepsilon\log y^{-1}.
\end{equation}

We now compute the leading term of the finite sum. Since $J$ is finite, the endpoints
$\{a_k,b_k:k\in J\}$
form a finite set. Choose $\delta>0$ such that the intervals of radius $2\delta$ centered at these endpoints are pairwise disjoint and contained in $I$.

Fix $k\in J$. Near $a_k$, write
$$
S_J(z)=\frac{\beta_{k,-}^2}{z-a_k}+H_{k,-}(z),
$$
where
$$
H_{k,-}(z)=-\frac{\beta_{k,+}^2}{z-b_k}
+\sum_{\substack{j\in J\\j\ne k}}\bigg(\frac{\beta_{j,-}^2}{z-a_j}
-\frac{\beta_{j,+}^2}{z-b_j}\bigg).
$$
Since all poles occurring in $H_{k,-}$ have positive distance from $a_k$, there exists a constant $C_{J}<\infty$ such that
$|H_{k,-}(x+iy)|\leq C_{J}$
whenever $|x-a_k|<\delta$ and $0<y<\delta$. Hence
$$
\bigg||S_J(x+iy)|-\frac{\beta_{k,-}^2}{\sqrt{(x-a_k)^2+y^2}}
\bigg|\leq C_{J}.
$$
Integrating over $(a_k-\delta,a_k+\delta)$ gives
\begin{equation}\label{leftpole}
    \int_{a_k-\delta}^{a_k+\delta}|S_J(x+iy)|\,dx=
\beta_{k,-}^2
\int_{-\delta}^{\delta}\frac{ds}{\sqrt{s^2+y^2}}+O_J(1)=
2\beta_{k,-}^2\log y^{-1}+O_J(1).
\end{equation}
The same argument at $b_k$ gives
\begin{equation}\label{rightpole}
\int_{b_k-\delta}^{b_k+\delta}|S_J(x+iy)|\,dx=
2\beta_{k,+}^2\log y^{-1}+O_J(1).
\end{equation}
On the complement of these endpoint neighborhoods inside $I$, all denominators occurring in $S_J$ are bounded away from zero. Hence
$$
\int_{I\setminus\bigcup_{k\in J}\big((a_k-\delta,a_k+\delta)\cup(b_k-\delta,b_k+\delta)\big)}
|S_J(x+iy)|\,dx
=O_J(1).
$$
Combining this estimate with \eqref{leftpole} and \eqref{rightpole}, we obtain
\begin{equation}\label{finest}
\int_I|S_J(x+iy)|\,dx=2\sum_{k\in J}\big(\beta_{k,-}^2+\beta_{k,+}^2\big)\log y^{-1}
+O_J(1).
\end{equation}

It remains to estimate the contribution outside $I$. Since $\mu$ is supported on $[b_0,a_0]$, for $x<b_0$ or $x>a_0$ away from the support one has
\begin{equation}\label{outside}
    \int_{\mathbb R\setminus I}|F_\mu(x+iy)|^2\,dx\leq
\int_{-\infty}^{b_0-\gamma_0}\frac{dx}{(b_0-x)^2}+
\int_{a_0+\gamma_0}^{\infty}\frac{dx}{(x-a_0)^2}=
\frac{2}{\gamma_0}
\end{equation}
uniformly for $0<y\leq\frac{1}{2}$.
By \eqref{tailest}, \eqref{finest}, and \eqref{outside},
$$
\int_{\mathbb R}|F_\mu(x+iy)|^2\,dx=
2\sum_{k\in J}\big(\beta_{k,-}^2+\beta_{k,+}^2\big)\log y^{-1}\\+
O_J(1)+O(\varepsilon\log y^{-1}).
$$
Dividing by $\log y^{-1}$ and letting $y\downarrow0$ gives
$$
\limsup_{y\downarrow0}
\bigg|
\frac{1}{\log y^{-1}}\int_{\mathbb R}|F_\mu(x+iy)|^2\,dx-2\sum_{k\in J}\big(\beta_{k,-}^2+\beta_{k,+}^2\big)
\bigg|
\leq C_*\varepsilon.
$$
By \eqref{tail}, we have
\begin{equation}\label{Fasym}
\lim_{y\downarrow0}
\frac{1}{\log y^{-1}}\int_{\mathbb R}|F_\mu(x+iy)|^2\,dx=2S_\mu.
\end{equation}
Substituting \eqref{Fasym},\eqref{Fzero} into \eqref{pl1} and using $|\widehat\mu(t)|^2$ is even, we obtain
\begin{equation}\label{Abel2}
\lim_{y\downarrow0}
\frac{1}{\log y^{-1}}
\int_0^\infty e^{-2yt}|\widehat\mu(t)|^2\,dt=\frac{1}{\pi}S_\mu.
\end{equation}

Define
$
H(T)=\int_0^T|\widehat\mu(t)|^2\,dt.
$ Then $H$ is nondecreasing and continuous, $H(0)=0$, and
its Laplace-Stieltjes transform is finite for every $s>0$,
since $|\widehat\mu|\leq1$. Setting $s=2y$ in
\eqref{Abel2} gives
$$
\int_0^\infty e^{-st}\,dH(t)
\sim\frac{S_\mu}{\pi}\log(1/s),
\qquad s\downarrow0.
$$
Karamata's Tauberian theorem
\cite[Theorem~1.7.1]{BGT}, with index zero and slowly
varying function $\log T$, now yields
$H(T)\sim\frac{S_\mu}{\pi}\log T, T\to\infty,$
which proves \eqref{avg}.

\end{proof}

\subsection{Sharp Pointwise Fourier decay}
We now turn to the pointwise Fourier transform. The upper bound follows from an $L^1$ estimate for the Borel transform, while the positivity of the averaged limit in \eqref{avg} shows that the exponent $1/2$ is sharp.

\begin{proof}[Proof of Theorem \ref{thm:fourier}]
We now prove the pointwise upper bound \eqref{Fdecay}. 
For $y>0$, let
$$
u_y(x)=\frac1\pi\operatorname{Im}F_\mu(x+iy)=\frac1\pi\int_{\mathbb{R}}\frac{y\,w_\mu(E)}{(x-E)^2+y^2}\,dE.
$$
Tonelli's theorem gives
\begin{equation}\label{equ:eq}
    \int_{\mathbb{R}}u_y(x)\,dx=\int_{\mathbb{R}}w_\mu(x)\,dx=1.
\end{equation}
We will prove
\begin{equation}\label{smooth}
\|w_\mu-u_y\|_{L^1(\mathbb{R})}\leq\frac{16}{\pi}\sqrt{y}\sum_{k\in\mathcal{K}_0}(\beta_{k,-}+\beta_{k,+}).
\end{equation}
Fix $x\in\Sigma$ at which \eqref{dens} holds. At a point where $w_\mu(x)>u_y(x)$, we have
$$
\begin{aligned}
\pi^2\big(w_\mu(x)-u_y(x)\big)^2
&\leq\pi^2\big(w_\mu(x)^2-u_y(x)^2\big)\\
&=\pi^2w_\mu(x)^2+\operatorname{Re}\big(F_\mu(x+iy)^2\big)-\big(\operatorname{Re}F_\mu(x+iy)\big)^2\\
&\leq\pi^2w_\mu(x)^2+\operatorname{Re}\big(F_\mu(x+iy)^2\big).
\end{aligned}
$$
 Consequently,
$$
\pi^2\big(w_\mu(x)-u_y(x)\big)_+^2\leq\big|\pi^2w_\mu(x)^2+\operatorname{Re}\big(F_\mu(x+iy)^2\big)\big|.
$$
Here the subscript $+$ denotes the positive part. 

For every real $d\ne0$,
$
\operatorname{Re}\frac1{d+iy}-\frac1d=-\frac{y^2}{d(d^2+y^2)}.
$
Applying this identity term by term to \eqref{pf} and \eqref{dens} and using absolutely  convergence, we obtain
$$
\begin{aligned}
\pi^2w_\mu(x)^2+\operatorname{Re}\big(F_\mu(x+iy)^2\big)
={}&-y^2\sum_{k\in\mathcal{K}_0}\frac{\beta_{k,-}^2}{(x-a_k)\big((x-a_k)^2+y^2\big)}\\
&+y^2\sum_{k\in\mathcal{K}_0}\frac{\beta_{k,+}^2}{(x-b_k)\big((x-b_k)^2+y^2\big)}.
\end{aligned}
$$
Taking absolute values and using $\sqrt{\sum_j A_j}\leq\sum_j\sqrt{A_j}$ for nonnegative $A_j$ yields
\begin{equation}\label{perr}
\begin{aligned}
\pi\big(w_\mu(x)-u_y(x)\big)_+
&\leq y\sum_{k\in\mathcal{K}_0}\frac{\beta_{k,-}}{\sqrt{|x-a_k|}\sqrt{(x-a_k)^2+y^2}}\\
&\quad+y\sum_{k\in\mathcal{K}_0}\frac{\beta_{k,+}}{\sqrt{|x-b_k|}\sqrt{(x-b_k)^2+y^2}}.
\end{aligned}
\end{equation}
This inequality also holds almost everywhere outside $\Sigma$, since $w_\mu=0$ there and $u_y\geq0$.

Integrating \eqref{perr}, using
$$
\int_{\mathbb{R}}\frac{y}{\sqrt{|x|}\sqrt{x^2+y^2}}\,dx
\leq\int_{|x|\leq y}\frac{dx}{\sqrt{|x|}}+y\int_{|x|>y}\frac{dx}{|x|^{3/2}}
=4\sqrt{y}+4\sqrt{y}=8\sqrt{y}
$$
and applying Tonelli's theorem give
$$
\int_{\mathbb{R}}(w_\mu-u_y)_+\,dx\leq\frac8\pi\sqrt{y}\sum_{k\in\mathcal{K}_0}(\beta_{k,-}+\beta_{k,+}).
$$
By \eqref{equ:eq},
$
\int_{\mathbb{R}}(w_\mu-u_y)_+\,dx=\int_{\mathbb{R}}(u_y-w_\mu)_+\,dx.
$
It follows that
$$
\|w_\mu-u_y\|_{L^1(\mathbb{R})}=2\int_{\mathbb{R}}(w_\mu-u_y)_+\,dx,
$$
which proves \eqref{smooth}.
Since $
 \widehat{w_\mu}(t)-\widehat{u_y}(t)=\big(1-e^{-y|t|}\big)\widehat\mu(t).
$
Using \eqref{smooth}, we have 
$$
\big(1-e^{-y|t|}\big)|\widehat\mu(t)|\leq\frac{16}{\pi}\sqrt{y}\sum_{k\in\mathcal{K}_0}(\beta_{k,-}+\beta_{k,+}).
$$
For $|t|\geq1$, choose $y=|t|^{-1}$, this proves the upper bound.

It remains to prove \eqref{sharp}. 
Let $L=\limsup_{t\to\infty}t^{1/2}|\widehat\mu(t)|.$
If $L=\infty$, there is nothing to prove. Otherwise, for every
$\varepsilon>0$, there exists $T_\varepsilon>0$ such that
$
|\widehat\mu(t)|^2\leq\frac{(L+\varepsilon)^2}{t},t\geq T_\varepsilon.
$
Hence, for $T>T_\varepsilon$,
$$
\int_0^T|\widehat\mu(t)|^2\,dt\leq
\int_0^{T_\varepsilon}|\widehat\mu(t)|^2\,dt+
(L+\varepsilon)^2
\int_{T_\varepsilon}^T\frac{dt}{t}\leq
T_\varepsilon+(L+\varepsilon)^2\log T.
$$
Dividing by $\log T$ and letting $T\to\infty$ gives
$
\frac1\pi
\sum_{k\in\mathcal K_0}
\bigl(\beta_{k,-}^2+\beta_{k,+}^2\bigr)\leq
(L+\varepsilon)^2.
$
Letting $\varepsilon\downarrow0$ yields the desired result.
\end{proof}

\section{Proofs of the quasiperiodic applications}
\subsection{Proof of Theorem \ref{thm:qp-holder} and Theorem \ref{thm:qp-fourier}}

\begin{proof}
As  $H_{v,\alpha,\theta}$  is subcritical, then for any $E\in \Sigma$, $L(\alpha,A_E)=0$, then by \cite[Theorem 2.2]{AvilaAMO}, $\mu_{\theta,n}$
is reflectionless on $\Sigma$ for every
$\theta\in\mathbb T$ (basically it follows from Kotani's theory \cite{Kotani} and compactness argument \cite{Remling}).

Let
$$
\mathcal K=\big\{k\in\mathbb Z\setminus\{0\}:
 G_k=(a_k,b_k)
\text{ is an open gap of }\Sigma
\big\}.
$$
 By the gap-labelling theorem \cite{JM},
$
N(E)=k\alpha\pmod{\mathbb Z}.
$

As $\alpha \in \mathrm{DC}$, by \cite[Theorem 1.6]{LYZZ}, the gaps are polynomially separated  
$$
\operatorname{dist}\big( G_k, G_l\big)\geq c_{\rm s}|k- l|^{-2\tau},
$$
while the gaps have exponential decay 
$\gamma_ l\leq C_ge^{-c_g| l|}$. 
It follows that 
$$
\log\mathcal C_k=\frac{1}{2}
\sum_{\substack{ l\in\mathcal K_0\\  l\neq k}}
\log\left(1+\frac{\gamma_ l}{\eta_{k l}}
\right)\leq\frac{1}{2}
\sum_{\substack{ l\in\mathcal K_0\\  l\neq k}}
\log\left(1+Ce^{-c_g| l|}|k- l|^{2\tau}\right)\leq
C\bigl(1+\log^{2}(2+|k|)\bigr).
$$
Therefore,
we obtain
$
\sum_{j\in\mathcal K_0} \sqrt{\gamma_j}\mathcal C_j<\infty.
$

Corollary~\ref{thm:holder} now applies and yields
$
\mu_{\theta,n}(X)\leq C_\Sigma|X|^{1/2}
$
for every measurable set $X$. Finally, let $P_\theta(X)$ be the spectral projection of $H_{ v,\alpha,\theta}$ associated with $X$. For $f\in   \ell^1(\mathbb Z)$,
$$
\begin{aligned}
\mu_{\theta,f}(X)^{1/2}=\|P_\theta(X)f\|_{\ell^2}&\leq\sum_{n\in\mathbb Z}|f(n)|\|P_\theta(X)\delta_n\|_{\ell^2}=\sum_{n\in\mathbb Z}|f(n)|\mu_{\theta,\delta_n}(X)^{1/2}\\
&\leq C^{1/2}|X|^{1/4}\|f\|_{   \ell^1}.
\end{aligned}
$$
Squaring proves the estimate in Theorem~\ref{thm:qp-holder}. Theorem \ref{thm:qp-fourier} follows from Theorem \ref{thm:fourier} and $
\beta_{0,\pm}(\theta)\geq\frac{1}{\sqrt{\gamma_0}\mathcal C_0}.
$
\end{proof}

\subsection{Proof of Theorem \ref{prop:liouville-obstruction}:}

\begin{proof}
Let $dN_{\lambda,\alpha}$ denote the density-of-states measure.  By
Avila-Last-Shamis-Zhou
\cite{AvilaLastShamisZhou}, the assumption
implies that  for every $s>0$,
\begin{equation}
\label{eq:ids-not-holder-liouville}
\sup_{0<|I|\leq1}\frac{dN_{\lambda,\alpha}(I)}{|I|^s}=\infty.
\end{equation}
Fix $s>0$.  For $M\in\mathbb N$, define
\begin{equation*}
\mathcal F_{s,M}=\left\{\theta\in\mathbb T:\nu_\theta(I)\leq M|I|^s
\text{ for every bounded open interval }I
\right\},
\end{equation*}
and let
$\mathcal U_{s,M}=
\mathbb T\setminus\mathcal F_{s,M}.
$ Since
$\theta\mapsto H_{\lambda,\alpha,\theta}$ is continuous in operator
norm, the measures $\nu_\theta$, depend
weakly continuously on $\theta$.  Therefore, for every open interval
$I$, the map
$\theta\longmapsto\nu_\theta(I)$ is lower semicontinuous.  It follows that $\mathcal F_{s,M}$ is
closed and $\mathcal U_{s,M}$ is open.

Set
\begin{equation}
\label{eq:Bs-definition}
\mathcal B_s=\bigcap_{M=1}^\infty\mathcal U_{s,M}.
\end{equation}
 Let
$P_I^\theta=
\mathbf 1_I(H_{\lambda,\alpha,\theta}).
$ Since
$\delta_2=H_{\lambda,\alpha,\theta}\delta_1-2\lambda\cos 2\pi(\theta+\alpha)\delta_1-
\delta_0,$
and $P_I^\theta$ commutes with
$H_{\lambda,\alpha,\theta}$, we have
\begin{equation*}
\left\|P_I^\theta\delta_2\right\|
\leq C_\lambda
\left(\left\|P_I^\theta\delta_0\right\|+\left\|P_I^\theta\delta_1\right\|
\right),
\end{equation*}
where $C_\lambda$ is independent of $\theta$ and $I$.  Hence
\begin{equation*}
\mu_{\theta,2}(I)
\leq C_\lambda
\left(\mu_{\theta,0}(I)+\mu_{\theta,1}(I)\right).
\end{equation*}
Solving the same recurrence for $\delta_0$ gives the reverse
estimate.  By covariance, after changing $C_\lambda$ if necessary,
\begin{equation}
\label{eq:nu-comparison}
C_\lambda^{-1}\nu_\theta(I)
\leq\nu_{\theta+\alpha}(I)\leq
C_\lambda\nu_\theta(I)
\end{equation}
for every $\theta\in\mathbb T$ and every bounded open interval $I$.

We claim that $\mathcal B_s$ is nonempty.  Suppose otherwise.  Then
$\mathbb T=
\bigcup_{M=1}^\infty
\mathcal F_{s,M}.
$ Since each $\mathcal F_{s,M}$ is closed, the Baire category theorem
implies that there exist $M_0\in\mathbb N$ and a nonempty open
interval $U\subset\mathbb T$ such that
$U\subset\mathcal F_{s,M_0}.$

Since the rotation by $\alpha$ is minimal, there exist finitely many
integers $k_1,\ldots,k_J$ such that
$\mathbb T=\bigcup_{j=1}^J(U-k_j\alpha).$ Hence, for every $\theta\in\mathbb T$, one can choose $j$ such that
$\theta+k_j\alpha\in U.$ 
Iterating \eqref{eq:nu-comparison}, we obtain
$\nu_\theta(I)\leq
C_\lambda^{|k_j|}
\nu_{\theta+k_j\alpha}(I)\leq
C_s|I|^s.$
Therefore
$dN_{\lambda,\alpha}(I)\leq C_s|I|^s,
$contradicting \eqref{eq:ids-not-holder-liouville}.  Hence
$\mathcal B_s\neq\varnothing.$
Choose $\theta_0\in\mathcal B_s$.  Since $\mathcal B_s$ is invariant
under the rotation by $\alpha$, we have $\mathcal B_s$ is dense.

By \eqref{eq:Bs-definition},
$\mathcal B_s\subset\mathcal U_{s,M}$ for every $M\in\mathbb N$.  Since $\mathcal B_s$ is dense, each
$\mathcal U_{s,M}$ is dense.  Thus
\begin{equation*}
\Theta:=\bigcap_{j=1}^\infty
\mathcal B_{1/j}=\bigcap_{j=1}^\infty\bigcap_{M=1}^\infty\mathcal U_{1/j,M}
\end{equation*}
 is a dense $G_\delta$ set.
Let $\theta\in\Theta$. By the definition of $\Theta$,  the first assertion
in \eqref{eq:liop} holds for every $s>0$.
Suppose that \eqref{eq:lioa} fails.  Then there exist
$s>0$, $C>0$, and $T_0>0$ such that
\begin{equation*}
\frac1T\int_0^T|\widehat\nu_\theta(t)|^2\,dt
\leq CT^{-s},
\qquad
T\geq T_0.
\end{equation*}
By \cite[Theorem 3.1]{Last}, we have $\nu_\theta$ is Holder continuous, contradicting
 the first assertion in \eqref{eq:liop}. Hence \eqref{eq:lioa} holds.

Finally, suppose that the second assertion in
\eqref{eq:liop} fails.  Then
there exist $s>0$, $C>0$, and $T_0>0$ such that
$|\widehat\nu_\theta(t)|\leq
C|t|^{-s},|t|\geq T_0.$
It follows that
\begin{equation*}
\frac1T\int_0^T
|\widehat\nu_\theta(t)|^2\,dt\leq C
\begin{cases}
T^{-2s}, & 0<s<\frac12,\\
T^{-1}\log(T), & s=\frac12,\\
T^{-1}, & s>\frac12.
\end{cases}
\end{equation*}
In each case, it contradicts \eqref{eq:lioa}. Hence the
second assertion in \eqref{eq:liop} holds.
\end{proof}

\subsection{Proof of Theorem~\ref{thm:qp-density}:}

Let $
\nu_\theta=\frac12\big(\mu_{\theta,0}+\mu_{\theta,1}\big),
d\nu_\theta(E)=w_\theta(E)\,dE.
$
Let $F_\theta$ be its Borel transform and let $\xi_k(\theta)\in[a_k,b_k]$ be the corresponding Dirichlet datum in the open gap $G_k=(a_k,b_k)$.  We  now obtain the density-one lower profile from the $L^1$ endpoint asymptotics.

\begin{proof}
The upper bound follows from Corollary~\ref{thm:holder}. Now we prove the lower bound estimate. By \cite[Footnote~1]{AvilaAMO}, $w_\theta(E)\geq\frac{1}{2\pi}$ for almost every $E\in\Sigma$. Therefore,
by Corollary \ref{prop:density-one1}, there is a unique zero $a_k<\xi_k(\theta)<b_k.$
Since $(H_{V,\alpha,\theta}-x)^{-1}$ depends
continuously on $\theta$, $F_\theta(x)$ is continuous in $\theta$.
We claim that $\theta\mapsto\xi_k(\theta)$ is continuous. Indeed, fix
$\theta_0\in\mathbb T$ and write $\xi_k(\theta_0)=\xi$. Given
$a_k<x_-<\xi<x_+<b_k$, strict monotonicity of $F_{\theta_0}$ on the gap gives
$F_{\theta_0}(x_-)<0<F_{\theta_0}(x_+).$  
By continuity in $\theta$, the same inequalities hold for all $\theta$
sufficiently close to $\theta_0$. Since $F_\theta$ is strictly increasing on
$(a_k,b_k)$ and vanishes at $\xi_k(\theta)$, it follows that
$x_-<\xi_k(\theta)<x_+.$
Letting $x_-\uparrow\xi$ and $x_+\downarrow\xi$ proves the continuity of
$\xi_k(\theta)$.
Therefore
$$
\beta_{k,-}(\theta)\geq\frac{\min_{\theta\in\mathbb T}\big(\xi_k(\theta)-a_k\big)}{\sqrt{\gamma_k}\sqrt{(a_k-b_0)(a_0-a_k)}\,\mathcal C_k}>0,\beta_{k,+}(\theta)
\geq
\frac{\min_{\theta\in\mathbb T}\big(b_k-\xi_k(\theta)\big)}
{\sqrt{\gamma_k}\sqrt{(b_k-b_0)(a_0-b_k)}\,\mathcal C_k}>0
$$
uniformly in $\theta$.  For $k=0$, the assertion follows directly from the
$
\beta_{0,\pm}(\theta)\geq\frac{1}{\sqrt{\gamma_0}\mathcal C_0}.
$
Therefore 
$$
h_k:=\frac1{2\pi}\inf_{\theta}
\min\{\beta_{k,-}(\theta),\beta_{k,+}(\theta)\}>0.
$$
Let $U_{k,\theta,\pm}$ be the sets in
Corollary~\ref{prop:density-one1} for $\mu=\nu_\theta$.
They have one-sided density one at the corresponding endpoints;
for fixed $k$, this is uniform in $\theta$ by the displayed
estimate.

Choose $\varepsilon_k>0$ with $2\sum_k\varepsilon_k\leq1/2$
and set $d_k=\varepsilon_k\min\{h_k,1\}$.
Then $\sum_kd_k<\infty$, and the defining inequalities give
$$
\sum_{k\in\mathcal K_0}d_k
\left(
\frac{\mathbf1_{U_{k,\theta,-}}(E)}{\sqrt{a_k-E}}
+
\frac{\mathbf1_{U_{k,\theta,+}}(E)}{\sqrt{E-b_k}}
\right)
\leq2\sum_k\varepsilon_k\,w_\theta(E)
\leq\frac12w_\theta(E)
$$
almost everywhere, with each summand set to zero outside
its indicated set. Combining this with
$w_\theta\geq(2\pi)^{-1}\mathbf1_\Sigma$ proves the lower bound.
\end{proof}

\subsection{Cancellation and bounded density}
\begin{proposition}\label{cancel}
Let $V\in C^\omega(\mathbb T,\mathbb R)$ be even and let $\alpha$ be Diophantine. Assume that every energy in the spectrum $\Sigma$ is subcritical and that $\Sigma$ is a Cantor set. Then, at the phase $\theta=0$, for every $k\in\mathcal K_0$ and every $r_0>0$,
$$
\operatorname*{ess\,inf}_{E\in\Sigma\cap(a_k-r_0,a_k)}
w_0(E)\sqrt{{a_k-E}}=
\operatorname*{ess\,inf}_{E\in\Sigma\cap(b_k,b_k+r_0)}
w_0(E)\sqrt{{E-b_k}}=0.
$$
\end{proposition}
\begin{proof}
Write $H=H_{V,\alpha,0}$. Let $H_\pm$ be its Dirichlet
restrictions to $\ell^2(\{\pm1,\pm2,\ldots\})$, and set
$m_\pm(z)=
\langle\delta_{\pm1},(H_\pm-z)^{-1}\delta_{\pm1}\rangle.$
The full-line Green matrix
$$M(z)=
\bigl(\langle\delta_i,(H-z)^{-1}\delta_j\rangle\bigr)_{i,j=0}^1$$
is given by the Schur complement formula
$$
M(z)=
\frac1{V(0)-z-m_-(z)-m_+(z)}
\begin{pmatrix}
1&-m_+(z)\\
-m_+(z)&m_+(z)\bigl(V(0)-z-m_-(z)\bigr)
\end{pmatrix};
$$
see \cite[Section~2.5]{Teschl}. In particular,
$$
F(z)=\frac12\operatorname{tr}M(z)
=\int_{\mathbb R}\frac{d\nu_0(E)}{E-z}.
$$

Since $V$ is even and the phase is zero, reflection identifies
the two half-line operators, so $m_+=m_-=:m$.
Moreover, normalized odd extension,
$$
(Uf)_0=0,\qquad
(Uf)_{\pm n}=\pm\frac{f_n}{\sqrt2},\qquad n\geq1,
$$
unitarily identifies $H_+$ with the restriction of $H$
to its reducing odd subspace. Thus $\sigma(H_+)\subset\Sigma$,
and $m$ extends analytically through every open gap.
Consequently,
$$
g(z):=M_{00}(z)=\frac1{V(0)-z-2m(z)}
$$
has no zero there. Taking the trace and eliminating $m$ gives
\begin{equation}\label{eq:fe}
F(z)=\frac{(z-V(0))^2+4}{8}\,g(z)-\frac1{8g(z)}.
\end{equation}

 By \cite[Footnote~1]{AvilaAMO}, $w_\theta(E)\geq\frac{1}{2\pi}$ for almost every $E\in\Sigma$,  thus Corollary \ref{prop:density-one1} shows each open gap $G_l$ contains
a zero $\xi_l$ of $F$. Hence
$$
|g(\xi_l)|=\frac1{\sqrt{(\xi_l-V(0))^2+4}}\asymp1
$$
uniformly in $l$. Set
$$
\delta_l=\min\{\xi_l-a_l,b_l-\xi_l\},
\qquad z_l=\xi_l+i\delta_l/2.
$$
 On the disk $D_l=\{z:|z-\xi_l|<\delta_l\}$, set
$h_l(z)=\frac{g(z)}{g(\xi_l)}.$
Since $g$ has constant nonzero sign on the real diameter,
$h_l$ is positive there. At nonreal points,
$$
\operatorname{Im}g(z)
=\operatorname{Im}z
\int_{\mathbb R}\frac{d\mu_{0,0}(E)}{|E-z|^2}\neq0,
$$
so $h_l(z)$ is nonreal. Thus $h_l$ avoids $(-\infty,0]$,
and its principal square root is analytic with positive
real part. Consequently,
$$
\Psi_l(z)=
\frac{\sqrt{h_l(z)}-1}{\sqrt{h_l(z)}+1}
$$
maps $D_l$ into the unit disk and satisfies
$\Psi_l(\xi_l)=0$. Schwarz's lemma gives
$
|\Psi_l(z_l)|
\leq\frac{|z_l-\xi_l|}{\delta_l}
=\frac12
,$ and consequently  $$
\frac13\leq|\sqrt{h_l(z_l)}|\leq3,
\qquad
\frac19\leq
\left|\frac{g(z_l)}{g(\xi_l)}\right|
\leq9.
$$
Together with $|g(\xi_l)|\asymp1$, this bounds both
$|g(z_l)|$ and $|g(z_l)|^{-1}$ uniformly in $l$.
Since the points $z_l$ remain in a bounded set,
\eqref{eq:fe} yields $|F(z_l)|\leq C$, uniformly in $l$.

Choose
$$
I_l=
\begin{cases}
(a_l-\delta_l,a_l),&\xi_l-a_l\leq b_l-\xi_l,\\
(b_l,b_l+\delta_l),&\xi_l-a_l>b_l-\xi_l.
\end{cases}
$$
For $E\in I_l$, $\delta_l\leq|E-\xi_l|\leq2\delta_l$.
It therefore yields
$$
\operatorname{Im}F(z_l)
\geq
\int_{I_l}
\frac{\delta_l/2}{(E-\xi_l)^2+\delta_l^2/4}\,d\nu_0(E)
\geq\frac{2}{17\delta_l}\nu_0(I_l),
$$
so $\nu_0(I_l)\leq C_1\delta_l$.
If $e_l$ is the chosen endpoint, the opposite half of
$(e_l-\delta_l,e_l+\delta_l)$ lies in $G_l$.
Homogeneity of the spectrum  \cite[Theorem 1.5]{LYZZ} gives
$$
|\Sigma\cap I_l|
=|\Sigma\cap(e_l-\delta_l,e_l+\delta_l)|
\geq c_1\delta_l.
$$
Set $C_0=2C_1/c_1$ and
$$
K_l=\{E\in\Sigma\cap I_l:w_0(E)\leq C_0\}.
$$
By Chebyshev's inequality,
$$
|(\Sigma\cap I_l)\setminus K_l|
\leq \frac{\nu_0(I_l)}{C_0}
\leq\frac{c_1\delta_l}{2},
\qquad
|K_l|\geq\frac{c_1\delta_l}{2}>0.
$$

Fix $k\in\mathcal K_0$ and $0<r<r_0$.
Since $\Sigma$ is perfect and nowhere dense (in the AMO case, see \cite{AJ3}), one can choose
an open gap $G_l$ such that
$[a_l-\gamma_l/2,b_l+\gamma_l/2]\subset(a_k-r,a_k).$
Indeed, choose a spectral point in $(a_k-r,a_k)$ and
distinct gaps converging to it; their lengths tend to zero.
Since $\delta_l\leq\gamma_l/2$, this gives
$K_l\subset\Sigma\cap(a_k-r,a_k)$ and hence
$$
0\leq
\operatorname*{ess\,inf}_{E\in\Sigma\cap(a_k-r_0,a_k)}
w_0(E)\sqrt{a_k-E}
\leq C_0\sqrt r.
$$
Letting $r\downarrow0$ proves the desired  assertion.
\end{proof}

\section*{Acknowledgements} 
We thank Dejun Feng and Zhiyuan Zhang for helpful discussions at an early stage of this work, and Zhongkai Tao for valuable discussions on Fourier decay that led to the proof of Proposition \ref{prop:avg}.
 Zhenfu Wang is supported by NSFC grant (124B2011) and Nankai Zhide Foundation.  Qi Zhou is supported by NSFC grant (12531006, 12526201) and Nankai Zhide Foundation.

\section*{Statement on the Use of AI}

ChatGPT 5.6 Sol and ChatGPT 6 were used for language editing, reference checking and   reviewing
the proofs for clarity and internal consistency. It was used at an early stage assist in identifying the phenomenon
established in Proposition~\ref{cancel}, the  proof was subsequently  reformulated, and fully rewritten by the authors based on the ideas of Theorem~\ref{thm:end}. The authors take full responsibility for all mathematical arguments and conclusions in the manuscript.

\end{document}